\documentclass[10pt,twoside,reqno]{amsart}
 \usepackage[top=1.4in, bottom=1.4in, left=1in, right=1in]{geometry}
\usepackage{graphicx}
\usepackage{amsmath,latexsym}
\usepackage{amsfonts,amssymb}
\usepackage{amsmath}
\usepackage{amscd, amsthm}
\usepackage{stmaryrd}
\usepackage{fancyhdr}
\usepackage{commath,enumitem}
\usepackage{indentfirst, hyperref}

\usepackage{pgf,tikz}
\usepackage{mathrsfs}
\usetikzlibrary{arrows,calc,positioning}
\usepackage{float}
\numberwithin{figure}{section}

\numberwithin{equation}{section}
\newtheorem{theorem}{Theorem}[section]
\newtheorem{lemma}[theorem]{Lemma}

\newtheorem{corollary}[theorem]{Corollary}
\newtheorem{definition}[theorem]{Definition}
\newtheorem{remark}[theorem]{Remark}

\let\olddefinition\definition
\renewcommand{\definition}{\olddefinition\normalfont}
\let\oldremark\remark
\renewcommand{\remark}{\oldremark\normalfont}

\newcommand{\A}{\mathbf{A}}
\newcommand{\C}{\mathbb{C}}

\newcommand{\F}{\mathcal{F}}
\renewcommand{\H}{\mathcal{H}}
\renewcommand{\L}{\mathbf{L}}
\newcommand{\N}{\mathbb{N}}
\renewcommand{\O}{\mathcal{O}}
\renewcommand{\P}{\mathcal{P}}
\newcommand{\R}{\mathbb{R}}
\newcommand{\T}{\mathbb{T}}
\newcommand{\Z}{\mathbb{Z}}

\newcommand{\diff}{\,\mathrm{d}}

\DeclareMathOperator{\hilbert}{\mathbf{H}}
\DeclareMathOperator{\sgn}{\mathrm{sgn}}
\DeclareMathOperator{\sech}{\mathrm{sech}}

\renewcommand{\vec}[1]{\mathbf{#1}}
\newcommand{\modDf}{|\partial_x|}
\newcommand{\Df}{\partial_x}
\newcommand{\Aop}{\mathbf{A}}
\newcommand{\Lop}{\L}
\newcommand{\Hi}{\mathbf{H}}
\newcommand{\Pbdry}{\partial\Omega}
\newcommand{\tf}{\mathbf{Q}}

\begin{document}\
\author{John K. Hunter}
\address{Department of Mathematics, University of California at Davis}
\email{jkhunter@ucdavis.edu}
\thanks{Supported by the NSF under grant number DMS-1616988}
\author{Jingyang Shu}
\address{Department of Mathematics, University of California at Davis}
\email{jyshu@ucdavis.edu}
\title[Fronts in the Surface Quasi-Geostrophic Equation]{Regularized and Approximate Equations for Sharp Fronts in the Surface Quasi-Geostrophic Equation and its Generalizations}
\begin{abstract}
We derive regularized contour dynamics equations for the motion of infinite sharp fronts in the two-dimensional
incompressible Euler, surface quasi-geostrophic (SQG), and generalized surface quasi-geostrophic (gSQG) equations. We derive
a cubic approximation of the contour dynamics equation and prove the short-time well-posedness of the
approximate equations for generalized surface quasi-geostrophic fronts and weak well-posedness for
surface quasi-geostrophic fronts.
\end{abstract}
\date{\today}

\maketitle

\section{Introduction}

The generalized surface quasi-geostrophic (gSQG) equation
\begin{subequations}\label{sqg}\begin{eqnarray}
& \theta_t + \vec{u} \cdot \nabla \theta = 0, \label{sqg1}\\
& \vec{u} = \nabla^\perp (-\Delta)^{-\alpha/2} \theta, \label{sqg2}
\end{eqnarray}\end{subequations}
is a transport equation  in two space dimensions for an active scalar field $\theta(\vec{x},t)$, where $\vec{x} =(x,y)$.
The divergence-free transport velocity $\vec{u}$ is determined nonlocally from $\theta$ by \eqref{sqg2}, where
$\nabla^\perp = (-\partial_y, \partial_x)$ is the perpendicular gradient, and $0<\alpha\le2$ is a parameter.
If $\alpha =2$, then \eqref{sqg1}--\eqref{sqg2} is the streamfunction-vorticity formulation of
the two-dimensional incompressible Euler equations \cite{Ma}, while if $\alpha = 1$, then \eqref{sqg1}--\eqref{sqg2}
is the SQG equation. The gSQG equation is a natural generalization of these cases.

The SQG equation is an approximate equation for
quasi-geostrophic flows confined near a surface \cite{La,Ped}.
It also provides a useful two-dimensional model for singularity formation
in the three-dimensional incompressible Euler equations \cite{CoMaTa94a, CoMaTa94b, MaTa}. For further analysis of the SQG equation, see \cite{BuShVi, Mar, Res} and the references cited there.

Since $\theta$ is advected by a velocity field $\vec{u}$, the gSQG equation is compatible with piecewise constant solutions
in which $\theta$ takes only two distinct values $\theta_{+}$, $\theta_-$, so that
\[
\theta(\vec{x},t) = \begin{cases}\theta_+ & \vec{x} \in \Omega(t),\\ \theta_- & \vec{x} \in \Omega^c(t),\end{cases}
\]
for some domain $\Omega(t) \subset \R^2$.
Under suitable assumptions, one may determine $\vec{u}(\cdot,t)$ from the domain $\Omega(t)$ and
obtain closed contour-dynamics equations for the boundary $\Pbdry(t)$, which moves with velocity $\vec{u}$.

We distinguish two particular types of domains:
\begin{itemize}
\item[1.] Patches, whose boundary is a smooth, simple, closed curve diffeomorphic to the circle $\T$.
\item[2.] Half-spaces, whose boundary  is a smooth, simple curve diffeomorphic to $\R$ that divides $\R^2$ into two half-spaces.
\end{itemize}

In the first case of a patch, one can take $\theta(\cdot,t) = \chi_{\Omega(t)}$ where $\Omega(t)$ is bounded, and then
$\vec{u} = \nabla^\perp G\ast \theta$, where $G$ is the Green's function for $(-\Delta)^{\alpha/2}$,
which is the two-dimensional Riesz potential of order $\alpha$ if $0<\alpha<2$ \cite{Zi,Stein}, or the Green's function
for the (negative) Laplacian if $\alpha=2$ . The convolution
converges since $\theta$ has compact support, and one obtains well-defined contour dynamics equations for the motion of the patch.

The vortex patch problem for the two-dimensional Euler equations has been studied extensively, and the boundary remains
smooth globally in time \cite{BeCo, Che,Che1,Ma}.
SQG and gSQG patches with $\alpha \in [1, 2)$ are analyzed in \cite{ CoCoGa, Gan},
where the local existence and uniqueness in Sobolev spaces
of solutions of a suitable parametric equation for the patch boundary is proved.
The formation of finite-time singularities in the boundary of an initially smooth SQG patch is an open question,
but  numerical solutions suggest that complicated, self-similar singularities can arise \cite{Dri}.
Singularities have been proved to occur for two gSQG patches with $\alpha$ sufficiently close to $2$
in the presence of a rigid boundary \cite{KiRyYaZl}.

In the second case of a half-space,  we refer to the boundary $\Pbdry(t)$ as a front, by which we
will always mean a sharp front across which $\theta$ is discontinuous. We will consider only
fronts that are a graph, located at
\begin{equation}\label{thetas}
y = \varphi(x,t),
\end{equation}
where $\varphi(\cdot,t) \colon \R \to \R$ is a smooth, bounded function. This assumption simplifies the evolution equations but
becomes invalid if the front breaks.

The front problem for vorticity discontinuities in the Euler equations is studied in \cite{BiHu,Ra}.
Local existence and uniqueness for spatially periodic SQG fronts is proved in \cite{Rod05}
for $C^\infty$-solutions by a Nash-Moser method,  and in \cite{FeRo11} for analytic solutions
by a Cauchy-Kowalewski method.  Almost sharp fronts are analyzed in \cite{CoFeRo, FeLuRo, FeRo12, FeRo15}, and
the global existence of Sobolev solutions for gSQG fronts with $0<\alpha<1$ that decay sufficiently rapidly as $|x|\to \infty$
is proved in \cite{CoGoIo}. If singularities do form in an SQG front, they would presumably differ from the ones
observed numerically in \cite{Dri} for an elliptical SQG patch; in those simulations, the patch forms a very thin
``neck'' which is not approximated by a half-space.

There is a difficulty in the formulation of contour dynamics for the half-space problem when $1\le \alpha \le 2$, which is
the main case of interest here.
The scalar field $\theta$ is not compactly supported, and the formal contour dynamics equations diverge
because the corresponding Green's function decays too slowly at infinity.
In this paper, we propose regularized contour dynamics equations for the motion of a front that
are obtained by introducing a large-distance cutoff parameter $\lambda$ in the contour dynamics
equations together with a suitable cutoff-dependent Galilean transformation with velocity $v(\lambda)$, where
$v(\lambda) \to \infty$ as $\lambda\to \infty$.

We show that the cut-off, Galilean transformed
contour dynamics equations have  a well-defined limit  as $\lambda\to \infty$, in which
the Galilean transformation removes a divergent constant from the velocity field $\vec{u}$.
The derivation is formal, in the sense that we do not attempt to show that solutions of the truncated equations
approach solutions of the regularized equations as $\lambda\to \infty$; rather, the goal is to formulate meaningful
contour dynamics equations for fronts that provide a starting point for further analysis.

After a normalization of the jump in $\theta$ across the front, the resulting equation for the displacement \eqref{thetas} of the front is
\begin{equation}\label{nonconseqn}
\varphi_t(x, t) + \int_\R \left[\varphi_x(x, t) - \varphi_x(x + \eta, t)\right] \biggl\{G(\eta)
- G\left(\sqrt{\eta^2 + \left[\varphi(x, t) - \varphi(x + \eta, t)\right]^2}\right)\biggr\} \diff{\eta} + \L \varphi_x(x, t) = 0,
\end{equation}
where
\begin{equation}
G(x) = \begin{cases} -\frac{1}{2\pi} \log |x| & \text{if $\alpha = 2$},
\\
1/|x|^{2-\alpha}  & \text{if $0<\alpha <2$}.
\end{cases}
\label{defG}
\end{equation}
The linear operator $\L$ in \eqref{nonconseqn} is given by
\begin{align*}
\Lop &=  \begin{cases}
\frac{1}{2}\modDf^{-1} & \text{if $\alpha = 2$ (Euler)},
\\
b_\alpha \modDf^{1-\alpha} & \text{if $0<\alpha < 1$ or $1 < \alpha < 2$},
\\
-2 \log \modDf & \text{if $\alpha = 1$ (SQG)},
\end{cases}
\end{align*}
where $\modDf = (-\Df^2)^{1/2}$ has symbol $|k|$, $\log \modDf$ has symbol $\log|k|$, and
\begin{equation}
b_\alpha =
2\sin \left(\frac{\pi\alpha}{2}\right) \Gamma(\alpha - 1).
\label{a-const}
\end{equation}
The integral in \eqref{nonconseqn} converges since
\[
G\left(\sqrt{\eta^2 + \abs{\varphi(x, t) - \varphi(x + \eta, t)}^2}\right) - G(\eta) = \O\left(\frac{1}{|\eta|^{4-\alpha}}\right)
\qquad \text{as $|\eta|\to \infty$}.
\]
Equation \eqref{nonconseqn} has the conservative form \eqref{conseqn} and the Hamiltonian form
\eqref{hameqn}. The equation also applies to spatially periodic solutions, when it can be written as \eqref{pernonconseqn}.
The explicit equations for Euler, SQG,  gSQG are written out in Sections~\ref{front:euler}, \ref{front:sqg}, \ref{front:gsqg},
respectively.

To study the small-amplitude dynamics of fronts, we approximate the nonlinear term in \eqref{nonconseqn} by
cubic terms, which gives the equation
\begin{align}
&\varphi_t + \frac{1}{2}\partial_x\left\{\varphi^2 \Aop \varphi - \varphi\Aop\varphi^2
+\frac{1}{3}\Aop\varphi^3\right\} + \Lop \varphi_x = 0,
\label{sqg_eq}
\end{align}
where the linear operator $\Aop$ is proportional to $\Df^2 \Lop$,
\begin{align*}
\Aop &= \begin{cases}
\frac{1}{2}\modDf & \text{if $\alpha = 2$ (Euler)},
\\
c_\alpha \modDf^{3-\alpha} & \text{if $0<\alpha < 1$ or $1 < \alpha < 2$},
\\
\Df^2 \log\modDf  & \text{if $\alpha = 1$ (SQG)}.
\end{cases}
\end{align*}
Here, the constant $c_\alpha$ is given by \eqref{defC12} for $1<\alpha<2$ or \eqref{defC01}
for $0<\alpha<1$.

The approximate equations for Euler, SQG, gSQG are written out explicitly  in Sections~\ref{approx:euler}, \ref{subsec:sqg},
\ref{approx:gsqg}, respectively.
The approximate equation with $\alpha=2$ for vorticity fronts is the same as the one derived by a systematic, but formal,
multiple-scale expansion directly from the incompressible Euler equations in \cite{BiHu}.

In Theorem~\ref{th:alpha}, we show that the spatially-periodic initial value problem for the approximate equation \eqref{sqg_eq}
with $1<\alpha \le 2$ is well-posed for short times in the Sobolev space $\dot{H}^s(\T)$ for $s> 7/2-\alpha$.
The proof is hyperbolic in nature and makes no use of the lower-order dispersive term.
A similar method of proof applied to the approximate SQG front equation, or its dispersionless version, gives
the short-time weak well-posedness result stated in Theorem~\ref{sqglwpthm}, in which solutions may lose Sobolev derivatives over time.
A stronger well-posedness result for the approximate SQG equation that uses the dispersion and has no loss of derivatives will be proved in \cite{HSZ}, following the method of C\'{o}rdoba et al. \cite{CoGoIo} for the case $\alpha < 1$.

These results can be compared with those of
Rodrigo and Fefferman  \cite{FeRo11, Rod05} for much more regular $C^\infty$ or analytic fronts, and the results of
C\'{o}rdoba,  C\'{o}rdoba, and Gancedo  \cite{ CoCoGa, Gan} for
Sobolev solutions of suitably parametrized equations for bounded SQG patches, which do not lose derivatives.

An outline of this paper is as follows. In Section~\ref{sec:dim}, we discuss the scaling properties of the gSQG front problem, including the anomalous scaling of the SQG problem.  In Section~\ref{sec:reg}, we derive the regularized contour dynamics equation
\eqref{nonconseqn}, and in Section~\ref{sec:approx}, we show that its cubic approximation can be written as
\eqref{sqg_eq}. In Section~\ref{Sec-LWP}, we prove a short-time well-posedness result for \eqref{sqg_eq} with
$1<\alpha \le 2$ and a short-time weak well-posedness result for the approximate SQG equation \eqref{sqg_eq} with $\alpha=1$.
In Section~\ref{sec:nls}, we consider traveling waves and the NLS-approximation
for \eqref{sqg_eq}, and in Section~\ref{sec:num}, we present some numerical solutions of the approximate SQG equation that
appear to show the formation of oscillatory singularities. Finally, in the Appendix, we prove
some algebraic inequalities used in the well-posedness proofs.

\section{Dimensional analysis}
\label{sec:dim}

One reason for the interest of the front problem is that, unlike
the patch problem,
a planar front does not define any length scales, so it preserves the scaling properties of the gSQG equation.

Suppose that $\theta$ is a piecewise-constant, odd function of $y$ that jumps across a planar front $y=0$,
\[
\theta  = \begin{cases} \theta_0 &\text{if $y > 0$},
\\
-\theta_0 & \text{if $y < 0$}.\end{cases}
\]
Up to a constant dimensionless factor, the corresponding gSQG shear flow
$\vec{u} = \left(u(y),0\right)$ with $u = -\partial_y |\partial_y|^{-\alpha} \theta$ is given by
\begin{equation*}
u(y) = \begin{cases} \theta_0 |y| & \text{if $\alpha=2$},
\\
\theta_0 |y|^{\alpha-1} &\text{if $0<\alpha < 1$ or $1<\alpha < 2$},
\\
\theta_0 \log |y| & \text{if $\alpha=1$}.
\end{cases}
\end{equation*}
As illustrated in Figure~\ref{fig:shear}, this shear flow is piecewise linear for the Euler equation, and has
a logarithmic divergence in the tangential velocity on the front for the SQG equation. The tangential velocity
on the front is zero if $1<\alpha \le 2$, and diverges algebraically if $0<\alpha<1$.

\begin{figure}
\includegraphics[width=0.7\textwidth]{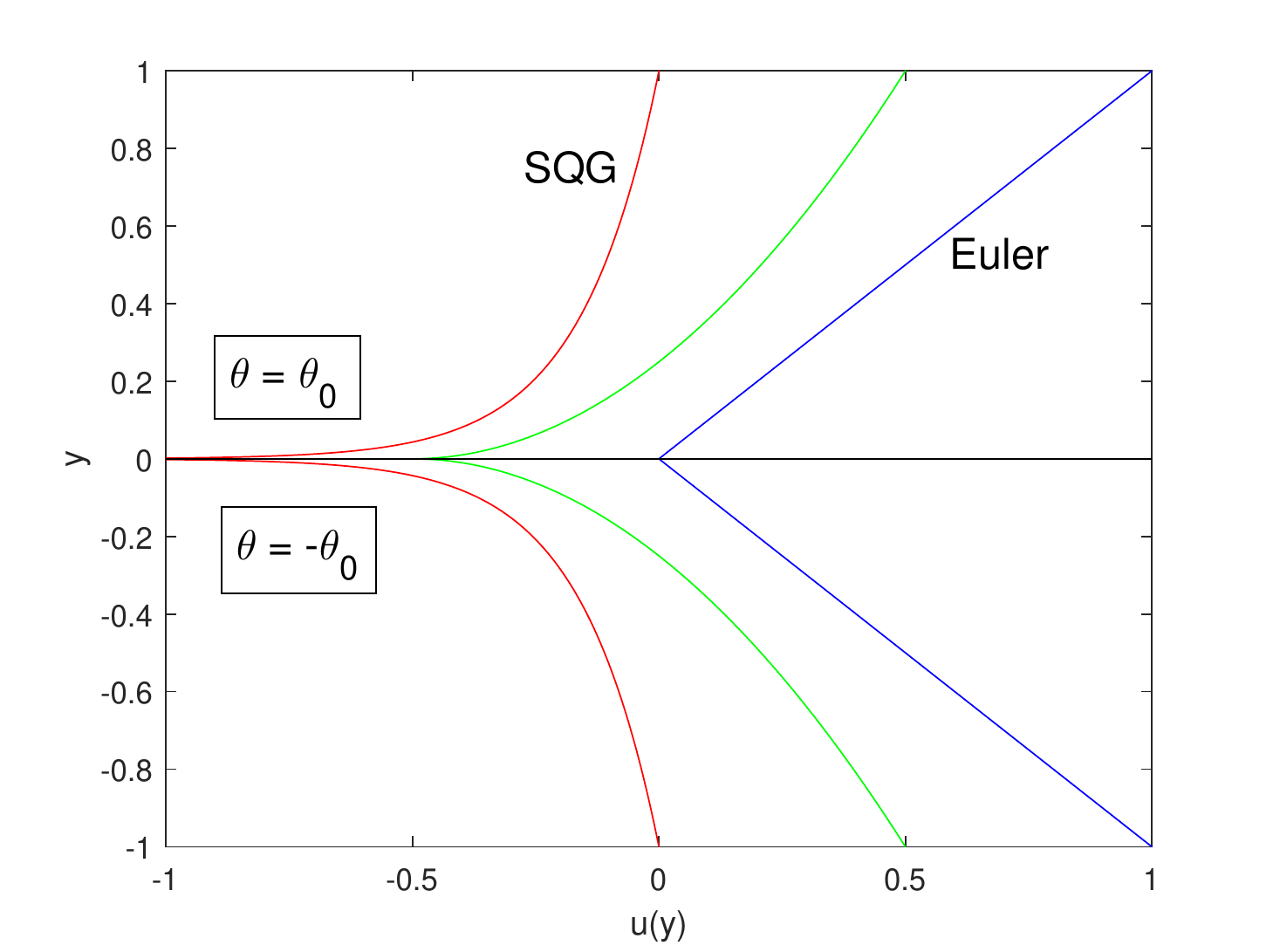}
\caption{The gSQG shear flow for a planar front with $\alpha = 1$ (red), $\alpha=3/2$ (green), and
$\alpha = 2$ (blue).}
\label{fig:shear}
\end{figure}

We denote the dimensions of a variable $f$ by $[f]$ and the dimensions of length and time by $L$ and $T$, respectively. Since $\vec{u} = \nabla^\perp(-\Delta)^{-\alpha/2} \theta$ is a velocity, we have that
\begin{equation*}
[\vec{u}] = \frac{L}{T},\qquad [\theta] = \frac{L^{2-\alpha}}{T}.
\end{equation*}
Thus, the vorticity $\theta$ has dimensions of frequency for the Euler equations ($\alpha=2$),
while $\theta$ has dimensions of velocity for the SQG equation ($\alpha=1$).

The front is linearly stable and waves propagate along it.
For small-amplitude, harmonic perturbations in the displacement of the form
$y = A e^{ik x-i\omega t} + \text{c.c.}$,
a naive dimensional argument gives the linearized dispersion relation
\begin{align*}
\omega &= C \theta_0 (\sgn k)|k|^{2-\alpha}
\end{align*}
where $C$ is a dimensionless constant.
A more detailed analysis verifies this dispersion relation for $0<\alpha <1$ and $1<\alpha \le 2$.
For example, in the case of vorticity discontinuities with $\alpha=2$, the waves are nondispersive with constant frequency \cite{BiHu},
while the waves are dispersive for $0<\alpha <1$ or $1<\alpha < 2$.

For the SQG equation with $\alpha=1$, the only parameter $\theta_0$ is a velocity, so one might expect that the waves on an SQG front are nondispersive with constant linearized phase speed. However, as observed by Rodrigo \cite{Rod05},
one finds that the linearized dispersion relation has the form
\[
\omega = c_0 k \log |k|,\qquad c_0 = C \theta_0
\]
with an additional factor that is  logarithmic in the wavenumber $k$.
Thus, the linearized SQG problem has an anomalous scaling invariance $(x,t)\mapsto (\tilde{x},\tilde{t})$
which is given for $\lambda > 0$ by
\[
\tilde{x}  = \lambda\left[x-c_0 (\log \lambda)t\right],\qquad  \tilde{t} = \lambda t.
\]
This invariance combines a hyperbolic-type scale-invariance  $(x,t)\mapsto (\lambda x,\lambda t)$ with a Galilean transformation
$(x,t)\mapsto (x-c_0(\log \lambda) t,  t)$.

For the nonlinear SQG front equation, \eqref{nonconseqn} with $\alpha=1$, one has $c_0=-2$. Moreover, the displacement
$\varphi$ has dimension $[\varphi] = L$, and the equation
is invariant under the transformation $(x,t,\varphi)\mapsto (\tilde{x},\tilde{t}, \tilde{\varphi})$ with
\begin{equation}
\tilde{x}  = \lambda\left[x + 2 (\log \lambda)t\right],\qquad \tilde{t} = \lambda t,\qquad\tilde{\varphi} = \lambda \varphi.
\label{sqg_scaling}
\end{equation}
We remark that the corresponding similarity solutions   have the form
\[
\varphi(x,t)  = t f\left(\frac{x}{t} - 2 \log t\right)
\]
rather than the usual power-law form for scale-invariant equations.

\section{Regularized contour dynamics equations for fronts}
\label{sec:reg}

In this section, we derive a regularized contour dynamics equation for infinite gSQG fronts. We begin by recalling the derivation of the
contour dynamics equations for bounded patches (see e.g., \cite{CoCoGa, Gan}).

\subsection{Contour dynamics for patches}

Suppose that $\Pbdry(t)$ is a smooth, simple, closed curve with bounded interior $\Omega(t)\subset \R^2$ and
\begin{equation}\label{thetasp}
\theta(\vec{x},t) = \begin{cases}\theta_0 & \vec{x} \in \Omega(t),\\ 0 & \vec{x} \in \Omega^c(t).\end{cases}
\end{equation}
The Green's function for the operator $(-\Delta)^{\alpha/2}$ on $\R^2$ is given by $g_\alpha G(|\vec{x}|)$ where \cite{Stein}
\begin{align*}
G(x) &= \begin{cases}
 -\frac{1}{2 \pi}\log|x| &  \text{if $\alpha = 2$},
 \\
 |x|^{-(2 - \alpha)} & \text{if $0 < \alpha < 2$},
\end{cases}
\qquad
 g_\alpha = \begin{cases} 1 & \text{if $\alpha = 2$},
 \\
 \frac{\Gamma(1 - \frac{\alpha}{2})}{2^\alpha \pi \Gamma(\frac{\alpha}{2})} & \text{if $0 < \alpha < 2$}.
 \end{cases}
\end{align*}
We normalize constants by choosing $\theta_0 = 1/g_\alpha$.
Then, using \eqref{sqg2} and Green's theorem, one finds that the velocity field corresponding to \eqref{thetasp} is
\begin{eqnarray}
\vec{u}(\vec{x}, t)
& = & \int_{\Pbdry(t)} G(|\vec{x}-\vec{x}'|) \vec{n}^\perp(\vec{x}', t) \diff{s(\vec{x}')}, \label{fl-u}
\end{eqnarray}
where $\vec{n} = (m,n)$ is the inward unit normal to $\Omega(t)$,
$\vec{n}^\perp = (-n,m)$, and $s(\vec{x}')$ is arc-length on $\Pbdry(t)$.

We suppose that $\Pbdry(t)$ is given by the parametric equation $\vec{x} = \vec{X}(\gamma,t)$, where $\vec{X}(\cdot,t) \colon \T\to \R^2$.
Since $\theta$ satisfies the transport equation \eqref{sqg1}, the curve $\Pbdry(t)$ moves with normal velocity
$\vec{X}_t\cdot\vec{n} = \vec{u}\cdot\vec{n}$.
If $0<\alpha \le 1$, then the tangential component of \eqref{fl-u} is unbounded on $\Pbdry(t)$, but the normal component
is well-defined, and the motion of the curve is determined solely by its normal velocity. The equation for $\vec{X}$ is therefore
\begin{equation}\label{cde-cpt}
\vec{X}_t(\gamma, t) = c(\gamma, t) \vec{X}_\gamma(\gamma, t) + \int_\T G(|\vec{X}(\gamma,t) - \vec{X}(\gamma', t)|) \left[\vec{X}_{\gamma}(\gamma, t) - \vec{X}_{\gamma'}(\gamma',t)\right] \diff{\gamma'},
\end{equation}
where $c(\cdot,t) \colon \T \to \R$ is an arbitrary smooth function that corresponds to a time-dependent reparametrization of the curve.
The inclusion of the term proportional to the tangent vector $\vec{X}_\gamma$ in the integral ensures that
the integral converges for $0<\alpha\le 1$; this term is not required for $1<\alpha \le 2$, since $G(|\vec{X} - \vec{X}'|)$ is locally
integrable, and it could be absorbed into $c$ in that case.

If $1\le \alpha \le 2$, there is a difficulty in extending the
contour dynamics equation to an infinite front $y=\varphi(x,t)$ where  $\varphi(\cdot,t) \colon \R \to \R$.
In that case, $\vec{X}(x,t) = \left(x,\varphi(x,t)\right)$, and we get formally
from \eqref{cde-cpt} that $c=0$ and
\begin{align}
\varphi_t(x, t) & = \int_\R G\left(\sqrt{(x - x')^2 + (\varphi(x, t) - \varphi(x', t))^2}\right) [\varphi_x(x, t) - \varphi_{x'}(x', t)] \diff{x'}.
\label{Req}
\end{align}
This equation makes sense for $0<\alpha < 1$ if $\varphi$ is a smooth, rapidly decaying (or bounded) function of $x$,
since the integral converges at $x=x'$ and at infinity \cite{CoGoIo}. However, it does not make sense for $1 \leq \alpha \leq 2$, since $G$ is not integrable at infinity and $\varphi_x(x, t)$ does not decay as $x'\to \infty$.

Roughly speaking, we have to regularize a short-distance ``ultraviolet'' singularity when $0<\alpha\le 1$,
caused by the infinite tangential velocity on the front, and a long-distance
``infrared'' singularity when $1\le \alpha\le 2$, caused by the slow decay of the Green's function.
The SQG equation --- which is the primary case of interest here ---
is peculiar in that it exhibits both infrared and ultraviolet singularities.

To regularize the
long-distance singularity, we introduce a long-range cutoff parameter $\lambda$, make a Galilean transformation
into a reference frame moving with  a suitable velocity $v(\lambda)$, where $v(\lambda)\to\infty$ as $\lambda\to\infty$,
and take the limit $\lambda\to\infty$.

The need for a Galilean transformation to get a well-defined limit can be seen directly in the case of the Euler equations. For example,
suppose one regards a planar vorticity discontinuity as the limit of a flow in a wide channel $-h <y<h$ as $h\to \infty$.
If one requires that the tangential flow on the channel boundaries $y=\pm h$ is equal to zero, then the corresponding
shear flow is $\vec{u} = (u(y),0)$ with $u(y) = \theta_0 (|y| - h)$.
Thus, one needs to make a Galilean transformation $x\mapsto x+\theta_0 h t$ in order to get a well-defined limit as $h\to \infty$.
This regularization would lead to the same equations as the ones derived here for a nonplanar front, but it appears to be more complicated to implement.

For the SQG equations, we have $\vec{u} =  R^\perp \theta$, where $R^\perp$ is the perpendicular Riesz transform.  The Riesz transform of an $L^{\infty}$-function belongs, in general, to $\text{BMO}$, and $\text{BMO}$-functions are only defined modulo an additive constant. Thus, an alternative regularization procedure for SQG fronts would be to determine $\vec{u}\in \text{BMO}(\R^2)$
modulo a constant and derive the contour dynamics equations from that velocity.
This procedure would presumably lead to equivalent equations to the ones derived here, but it also appears to be more complicated to implement.

\subsection{Cutoff Regularization}

We consider a front $y=\varphi(x,t)$ across which $\theta$ jumps from $\theta_0/2$ to $-\theta_0/2$.
After a change of variables $x'=x+\eta$ in \eqref{Req}, we introduce a large cutoff parameter $\lambda > 0$
to get the truncated  equation
\begin{equation}\label{lambda-eveqn0}
\varphi_t(x, t) = \int_{-\lambda}^{\lambda}G\left(\sqrt{\eta^2
+ (\varphi(x, t) - \varphi(x + \eta, t))^2}\right) \left[\varphi_x(x, t) - \varphi_x(x + \eta, t)\right] \diff{\eta}.
\end{equation}
We assume that  $\varphi(\cdot,t) \colon \R \to \R$ is a smooth bounded function
with bounded first derivative. The integral in \eqref{lambda-eveqn0} converges since $\eta G(\eta)$
is locally integrable for $\alpha > 0$ when $G(\eta)$ is given by \eqref{defG}.

It is convenient to write \eqref{lambda-eveqn0} in the conservative form
\begin{equation}\label{lambda-eveqn1}
\varphi_t(x, t) = \partial_x \int_{-\lambda}^{\lambda} F\left(\eta, \varphi(x, t) - \varphi(x + \eta, t)\right) \diff{\eta},
\end{equation}
where $F$ is defined by
\begin{equation}
F(x, y) = \int_0^y G\left(\sqrt{x^2 + s^2}\right) \diff{s}.
\label{defF}
\end{equation}
To take the limit $\lambda\to \infty$,
we write \eqref{lambda-eveqn1} as
\begin{align}
\begin{split}
&\varphi_t(x, t) + \partial_x  \int_{-\lambda}^{\lambda} K(\eta, \varphi(x, t) - \varphi(x + \eta, t))  \diff{\eta}
- \partial_x  \int_{-\lambda}^{\lambda}
 G(\eta)\left[ \varphi(x, t)  - \varphi(x + \eta, t) \right] \diff{\eta}
= 0,
\end{split}
\label{lambda-eveqn2}
\end{align}
where
\begin{align}
\begin{split}
K(x, y) & = G(x)y - F(x, y).
\end{split}
\label{defK}
\end{align}

First, we consider the nonlinear term in \eqref{lambda-eveqn2}. We find from \eqref{defF} that
\begin{equation*}
F(x, y) \sim  G(x)y + \O\left(\frac{G'(x)}{x}\right)
\qquad
\text{as $\abs{x} \to \infty$ with $y$ fixed},
\]
so when $G$ is given by \eqref{defG} and $\varphi(\cdot,t)$ is bounded, we have
\begin{equation}
K\left(\eta,\varphi(x, t) - \varphi(x + \eta, t)\right) = \O\left(\frac{1}{|\eta|^{4-\alpha}}\right)
\qquad \text{as $|\eta| \to \infty$}.
\label{decayK}
\end{equation}
It follows that
\[
\lim_{\lambda\to\infty}  \int_{-\lambda}^{\lambda} K\left(\eta, \varphi(x, t) - \varphi(x + \eta, t)\right)  \diff{\eta}
=  \int_{\R} K(\eta, \varphi(x, t) - \varphi(x + \eta, t))  \diff{\eta},
\]
since the integral converges on $\R$.

Next, we consider the linear term
\[
\L_\lambda \varphi(x,t) = -\int_{-\lambda}^{\lambda}
 G(\eta)\left[\varphi(x, t) - \varphi(x + \eta, t)\right] \diff{\eta},
 \]
where $\eta G(\eta)$ is locally integrable.
 There are three cases, depending on whether the Green's function $G(\eta)$ is: (a) nonintegrable at $0$ and
 integrable at infinity($0<\alpha<1$);
 (b) integrable at $0$ and nonintegrable at infinity ($1< \alpha \le 2$); (c) nonintegrable at both $0$ and $\infty$ ($\alpha =1$).
 We consider each of them in turn.

 In case (a), we have $\L_\lambda \varphi \to \L\varphi$ as $\lambda \to \infty$, where
 \begin{equation*}
 \L \varphi(x,t) = -\int_{\R} G(\eta)\left[\varphi(x, t) - \varphi(x + \eta, t)\right] \diff{\eta}.
 \end{equation*}
 This operator is translation invariant; its
 symbol $b(k)$, such that $\L e^{ik x} = b(k) e^{ik x}$, is the function
 \begin{equation}
 b(k) = -\int_{\R} G(\eta)\left(1 - e^{ik\eta}\right) \diff{\eta}.
 \label{defLa}
 \end{equation}
 Thus, the limit of \eqref{lambda-eveqn2} as $\lambda\to \infty$ is
 \begin{equation}\label{conseqn}
\varphi_t(x, t) + \partial_x \int_{\R} K(\eta, \varphi(x, t) - \varphi(x + \eta, t)) \diff{\eta} + \L \varphi_x(x, t) =0.
\end{equation}
Taking the $x$-derivative under the integral sign in \eqref{conseqn} and using \eqref{defK}, we get the non-conservative
form of the regularized equation in \eqref{nonconseqn}.

 In case (b), we write
 \begin{align*}
 \L_\lambda \varphi(x,t)
 &= v(\lambda)\varphi(x,t) + \int_{-\lambda}^{\lambda}
 G(\eta) \varphi(x + \eta, t) \diff{\eta},
 \end{align*}
where
\begin{align}
 v(\lambda) = - 2\int_{0}^\lambda G(\eta) \diff{\eta},
 \label{defvb}
 \end{align}
 which diverges as $\lambda\to\infty$. Then \eqref{lambda-eveqn2} becomes
\begin{align*}
\begin{split}
&\varphi_t(x, t) + v(\lambda)\varphi_x(x,t)
+ \partial_x  \int_{-\lambda}^{\lambda} K(\eta, \varphi(x, t) - \varphi(x + \eta, t))  \diff{\eta}
+  \partial_x \int_{-\lambda}^{\lambda}  G(\eta) \varphi(x + \eta, t) \diff{\eta} = 0.
\end{split}
\end{align*}
We make a Galilean transformation $x \mapsto x - v(\lambda)t$ into a reference frame moving with velocity
$v(\lambda)$, which removes the term $v\varphi_x$, and then let $\lambda\to \infty$, which gives
the regularized equation \eqref{conseqn} with
\begin{equation*}
\L\varphi(x,t) =  \int_{\R} G(\eta) \varphi(x + \eta, t) \diff{\eta}.
\end{equation*}
This integral converges if $\varphi(\cdot,t)$ decays sufficiently rapidly at infinity, and can be interpreted
in a distributional sense in other cases.
The symbol of $\L$,
\begin{equation}
b(k) = \int_\R G(\eta) e^{ik \eta} \, d\eta,
 \label{defLb}
\end{equation}
is well-defined as a tempered distribution since
$G(\eta)$ is locally integrable and has, at most, slow growth as $|\eta|\to\infty$.

In case (c), we write
\begin{align*}
 \L_\lambda \varphi(x,t) &= v(\lambda) \varphi_x(x,t)
 + \int_{1<|\eta| <\lambda} G(\eta) \varphi(x + \eta, t) \diff{\eta}
 - \int_{|\eta| <1} G(\eta)\left[\varphi(x, t) - \varphi(x + \eta, t)\right] \diff{\eta},
\end{align*}
where
\begin{equation}
 v(\lambda) = -2\int_{1}^{\lambda} G(\eta) \diff{\eta}.
 \label{defvc}
 \end{equation}
Making a Galilean transformation $x\mapsto x-v(\lambda) t$, and taking the limit of the resulting equation as $\lambda \to \infty$,
we get the regularized equation \eqref{conseqn} with
 \begin{equation*}
 \L \varphi(x,t) = \int_{|\eta| >1} G(\eta) \varphi(x + \eta, t) \diff{\eta}
 - \int_{|\eta| <1} G(\eta)\left[\varphi(x, t) - \varphi(x + \eta, t)\right] \diff{\eta}.
 \end{equation*}
 In this case, the symbol of $\L$ is the sum of a tempered distribution and a function,
 \begin{equation}
 b(k) = \int_{|\eta| >1} G(\eta) e^{ik\eta}\diff{\eta}
 -\int_{|\eta| <1} G(\eta)\left[1 - e^{ik\eta}\right] \diff{\eta}.
\label{defLc}
\end{equation}

\subsection{Regularized front equations}

In this section, we write out the specific form of the regularized front equations for the Euler, SQG, and gSQG
equations.

\subsubsection{Euler equation $(\alpha = 2)$}
\label{front:euler}

The Green's function for the Euler equation is
\[G(x) = -\frac{1}{2\pi} \log|x|.\]
It follows from \eqref{defF} and \eqref{defK} that
\[\begin{aligned}
F(x, y)
& = -\frac{1}{2\pi} \left\{y \log\sqrt{x^2 + y^2} + x \tan^{-1}\left(\frac{y}{x}\right) - y\right\},
\\
K(x,y) &= \frac{1}{2\pi} \left\{y \log\left[\frac{\sqrt{x^2 + y^2}}{|x|}\right]
+ x \tan^{-1}\left(\frac{y}{x}\right) - y\right\}.
\end{aligned}\]
In addition,  the velocity \eqref{defvb} used in the Galilean transformation is
\[
v(\lambda) = \frac 1\pi \left(\lambda \log\lambda - \lambda\right).
\]

Using the distributional Fourier transform of the logarithm \cite{Vlad}, we get from \eqref{defLb} that
the symbol of $\L$ is given by
\begin{align*}
b(k) &= -\frac{1}{2\pi}\F\left[\log|\cdot|\right](-k)
= \gamma \delta(-k) + \frac{1}{2} \text{p.f.} \frac{1}{|k|},
\end{align*}
where $\gamma$ is the Euler-Mascheroni constant.
It follows that $\L\partial_x = -\frac{1}{2}\Hi$ where $\Hi$ is the Hilbert transform with symbol $-i\sgn k$.

Thus, the regularized equation for vorticity fronts is
\[\begin{aligned}
\varphi_t(x, t) & + \frac{1}{2 \pi} \partial_x \int_\R \bigg\{[\varphi(x, t) - \varphi(x + \eta, t)]
\log\left[\frac{\sqrt{\eta^2 + [\varphi(x, t) - \varphi(x + \eta, t)]^2}}{|\eta|}\right]\\
& \quad + \eta \tan^{-1} \left(\frac{\varphi(x, t) - \varphi(x + \eta, t)}{\eta}\right)
- [\varphi(x, t) - \varphi(x + \eta, t)]\bigg\} \diff{\eta} = \frac 12 \hilbert \varphi(x, t),
\end{aligned}\]
and the non-conservative form of the equation is
\[
\varphi_t(x, t) + \frac{1}{2\pi} \int_\R [\varphi_x(x, t) - \varphi_x(x + \eta, t)]
 \log\left[\frac{\sqrt{\eta^2 + [\varphi(x, t) - \varphi(x + \eta, t)]^2}}{|\eta|}\right]\diff{\eta} = \frac 12 \hilbert \varphi(x, t).
\]

\subsubsection{SQG equation $(\alpha = 1)$}
\label{front:sqg}

The Green's function for the SQG equation  is
\[
G(x) = \frac{1}{|x|}.
\]
It follows from \eqref{defF} and \eqref{defK}  that
\begin{align*}
F(x, y) & = \sinh^{-1}\left(\frac{y}{\abs{x}}\right),
\qquad
K(x,y)  =\frac{y}{\abs{x}} - \sinh^{-1}\left(\frac{y}{\abs{x}}\right).
\end{align*}
In addition,  the velocity \eqref{defvc} used in the Galilean transformation is
\[
v(\lambda) = -2\log \lambda.
\]

We find from \eqref{defLc} that
\begin{align*}
b(k) &= 2\left(\int_1^\infty \frac{\cos (k\eta)}{\eta}\diff{\eta} - \int_0^1 \frac{1-\cos (k\eta)}{\eta} \diff{\eta} \right)
\\
&= v_1 - 2\log|k|,
\\
v_1 &= 2\left(\int_1^\infty \frac{\cos\eta}{\eta}\diff{\eta} - \int_0^1 \frac{1-\cos\eta}{\eta}\diff{\eta}\right).
\end{align*}
We can absorb $v_1$ into $v(\lambda)$ by the use of a Galilean transformation $x\mapsto x - [v_1 + v(\lambda)]t$, and
then the remaining part of the linear operator is $\L = -2\log \modDf$.

Thus, the  regularized equation for SQG fronts is
\[\varphi_t(x, t) + \partial_x \int_\R \bigg\{\frac{\varphi(x, t) - \varphi(x + \eta, t)}{\abs{\eta}}
- \sinh^{-1}\bigg[\frac{\varphi(x, t) - \varphi(x + \eta, t)}{\abs{\eta}}\bigg]\bigg\} \diff{\eta}
= 2 \log\abs{\partial_x} \varphi_x(x, t),\]
and the non-conservative form of the equation is
\[
\varphi_t(x, t) + \int_\R [\varphi_x(x, t) - \varphi_x(x + \eta, t)] \bigg\{\frac{1}{\abs{\eta}}
- \frac{1}{\sqrt{\eta^2 + [\varphi(x, t) - \varphi(x + \eta, t)]^2}}\bigg\} \diff{\eta} = 2 \log\abs{\partial_x} \varphi_x(x, t).
\]

\subsubsection{gSQG equation}
\label{front:gsqg}

The Green's function for the gSQG equation is
\[G(x) = \frac{1}{|x|^{2 - \alpha}}.\]
For $1 < \alpha < 2$, it follows from \eqref{defF} and \eqref{defK} that
\begin{align*}
F(x, y) &= \int_0^y \frac{1}{(x^2 + s^2)^{(2 - \alpha)/2}} \diff{s},
\qquad
K(x, y) = \frac{y}{|x|^{2 - \alpha}}
- \int_0^y \frac{1}{(x^2 + s^2)^{(2 - \alpha)/2}} \diff{s}.
\end{align*}
In addition, from \eqref{defvb}, the velocity used in the regularization is
\[
v(\lambda) =  -\frac{2 \lambda^{\alpha-1}}{\alpha-1}.
\]
We find from \eqref{defLb}, that
\begin{align*}
b(k) &
= b_\alpha\abs{k}^{1-\alpha},
\qquad
b_\alpha = 2\int_0^\infty \frac{\cos \eta}{\eta^{2-\alpha}}\diff\eta = 2 \sin\left(\frac{ \pi\alpha}{2}\right) \Gamma(\alpha-1),
\end{align*}
where $b_\alpha > 0$ is given by \eqref{a-const}.
The corresponding operator is $\L = b_\alpha \modDf^{1-\alpha}$.

Thus, the regularized equation for gSQG fronts with $1<\alpha<2$ is
\[\varphi_t(x, t)  + \partial_x \int_\R \Biggl\{\bigg[\frac{\varphi(x, t)
- \varphi(x + \eta, t)}{\abs{\eta}^{2 - \alpha}}\bigg] - F\left(k, \varphi(x, t) - \varphi(x + \eta)\right)\Biggr\} \diff{\eta}
+ b_\alpha \modDf^{1-\alpha} \varphi_x(x, t)= 0,\]
and the non-conservative form of the equation is
\begin{equation}\label{reg-1a2}\begin{aligned}
&\varphi_t(x, t) + \int_\R [\varphi_x(x, t) - \varphi_x(x + \eta, t)] \Biggr\{ \frac{1}{\abs{\eta}^{2 - \alpha}}
- \frac{1}{\left(\eta^2
 + [\varphi(x, t) - \varphi(x + \eta, t)]^2\right)^{(2- \alpha) / 2}} \Biggl\} \diff{\eta}
\\
&\qquad\qquad\qquad + b_\alpha \abs{\partial_x}^{1 - \alpha} \varphi_x(x, t) = 0.
\end{aligned}\end{equation}

The derivation of the regularized gSQG equation for $0 < \alpha < 1$ is similar to the case of $1 < \alpha < 2$,
except that we do not need to make a Galilean transformation to obtain a finite limit, and
we find $b(k)$ from \eqref{defLa}. One obtains the same equation
\eqref{reg-1a2} as in the case $1<\alpha<2$, where $b_\alpha<0$ is  given by \eqref{a-const}.
This equation agrees with the gSQG equation for $0 < \alpha < 1$ that is analyzed in \cite{CoGoIo}.

\subsection{Spatially periodic solutions}

The previous equations do not require that $\varphi(\cdot,t)$ is rapidly decreasing; in particular, they apply to smooth periodic solutions
$\varphi(\cdot,t) \colon \T \to \R$ where $\T = \R/2\pi\Z$.
The symbol of the linear operator $\L$ remains the same. Moreover,
we can write the nonlinear term in \eqref{conseqn}  as
\begin{align*}
 \int_{\R} K(\eta, \varphi(x, t) - \varphi(x + \eta, t))  \diff{\eta}
 &=\int_\T K_p(\eta, \varphi(x, t) - \varphi(x + \eta, t))  \diff{\eta},
\\
K_p(x,y) &= \sum_{n\in \Z}  K(x + 2n\pi, y).
\end{align*}
The sum defining $K_p$ converges because of \eqref{decayK}.
The conservative form of the periodic front equation is then
\begin{equation*}
\varphi_t(x, t) + \partial_x \int_{\T} K_p(\eta, \varphi(x, t) - \varphi(x + \eta, t)) \diff{\eta} + \L \varphi_x(x, t) =0.
\end{equation*}

The non-conservative form can be written as
\begin{align}
\begin{split}
&\varphi_t(x, t) +  \int_{\T}\left[\varphi_x(x, t) - \varphi_x(x + \eta, t)\right]
\biggl\{G_p(\eta,0)
- G_p\left(\eta, \varphi(x, t) - \varphi(x + \eta, t)\right)\biggr\}\diff{\eta}
\\
&\hskip2in+ \L \varphi_x(x, t) =0,
\end{split}
\label{pernonconseqn}
\end{align}
where
\begin{equation*}
G_p(x,y) = G \left(\sqrt{x^2 + y^2}\right) + \sum_{n\in\Z_*}\left[G \left(\sqrt{(x+2\pi n)^2 + y^2}\right) - G(2\pi n)\right]
\end{equation*}
is the Green's function of $(-\Delta)^{\alpha/2}$ on the cylinder $\T\times \R$, and $\Z_* = \Z\setminus\{0\}$.

One can verify that \eqref{pernonconseqn} is equivalent, up to a Galilean transformation, to the straightforward
contour dynamics equation on a cylinder,
\begin{align*}
\begin{split}
&\varphi_t(x, t)  -  \int_{\T}\left[\varphi_x(x, t) - \varphi_x(x + \eta, t)\right]
G_p\left(\eta,\varphi(x, t) - \varphi(x + \eta, t)\right)\diff{\eta} = 0.
\end{split}
\end{align*}
However, \eqref{pernonconseqn} explicitly separates the linear dispersive term from the cubic-order nonlinearity.

For the Euler equation with
\[G(\eta) = -\frac{1}{2\pi} \log|\eta|,\]
we get from the Euler product formula for $\sin z$ that
\begin{align*}
G_p(x,y) &= -\frac{1}{2\pi} \log \left|\sin\left(\frac{z}{2}\right)\right|,
\end{align*}
where $z=x+iy$.
For the periodic SQG front equation with $G(\eta)=1/|\eta|$, we have
\begin{align*}
G_p(x,y) &= \frac{1}{\sqrt{x^2 +y^2}} + \sum_{n \in \Z_*}\left\{ \frac{1}{\sqrt{(x+2\pi n)^2 +y^2}}
- \frac{1}{2\pi |n|}\right\}.
\end{align*}

\subsection{Hamiltonian structure}

Let $\H$ be a functional of $\varphi \colon \R \to \R$ of the form
\[
\H(\varphi) = \frac 12 \iint_{\R \times \R} H(x - x', \varphi - \varphi') \diff{x}\diff{x'},
\]
where $H(x,y)$ is an even function of $x, y\in \R$ and $\varphi = \varphi(x)$, $\varphi' = \varphi(x')$.
The variational derivative of $\mathcal{H}$ is given by
\[
\frac{\delta \H}{\delta \varphi(x)} = \int_\R K(x - x', \varphi - \varphi') \diff{x'},
\]
where  $K(x, y) = H_y(x, y)$.
Thus, the conservative front equation \eqref{conseqn} has the Hamiltonian form
\begin{equation}
\varphi_t + \partial_x \bigg[\frac{\delta \H}{\delta \varphi}\bigg]= 0,
\label{hameqn}
\end{equation}
where $\Df$ is the Hamiltonian operator and the Hamiltonian is
\begin{align*}
\H(\varphi) &= \frac 12 \iint_{\R \times \R} H\left(x - x', \varphi - \varphi'\right) \diff{x}\diff{x'}
+ \frac 12 \int_{\R} \varphi \L \varphi \diff{x},
\\
H(x,y) &= \int_0^y K(x,s)\diff{s}.
\end{align*}
The corresponding conserved momentum, which generates spatial translations, is
\[\P(\varphi) = \frac 12 \int_\R \varphi^2 \diff{x}.\]

\section{Approximate equation}
\label{sec:approx}

In this section, we derive the approximate equation \eqref{sqg_eq} for fronts with small slopes by truncating the nonlinearity
in the full equation \eqref{nonconseqn} at cubic terms.

It follows from \eqref{defF} that
\[\begin{aligned}
F(x, y) & = |x|\left\{G(x) \frac{y}{|x|} + \frac 16 x G'(x) \frac{y^3}{|x|^3}
 + \O\left(\frac{y^5}{\abs{x}^5}\right)\right\}
 \qquad \text{as}\quad \frac{y}{|x|} \to 0.
\end{aligned}\]
Retaining the lowest order terms in $y$, we find that the kernel $K$ in \eqref{defK} has the approximation
\[K(x, y) \sim - \frac{1}{6}\frac{G'(x)}{x}y^3.\]
Thus, the cubic approximation of the conservative equation \eqref{conseqn} is
\begin{equation}\label{consapprox}
\varphi_t(x, t) -\frac 16 \partial_x \int_\R \frac{G'(\eta)}{\eta}
\left [\varphi(x, t) - \varphi(x + \eta, t)\right]^3 \diff{\eta} + \L \varphi_x(x, t) = 0.
\end{equation}
Equation \eqref{consapprox} is equivalent to \eqref{sqg_eq}, as we show
by writing it in spectral form.

\subsection{Spectral equation}

For definiteness, we suppose that $\varphi \colon \R \to \R$ is a smooth function that
decreases sufficiently rapidly at infinity, with Fourier transform
\[
\hat{\varphi}(k) = \frac{1}{2\pi} \int_\R \varphi(x) e^{-ik x}\diff{x}.
\]
The same results apply to periodic functions $\varphi \colon \T \to \R$, with Fourier
transforms replaced by Fourier series.

Then
\begin{equation}
-\int_\R \frac{G'(\eta)}{\eta}\left [\varphi(x) - \varphi(x + \eta)\right]^3 \diff{\eta}
= \int_{\R^3} T(k_2, k_3, k_4) \hat{\varphi}(k_2) \hat{\varphi}(k_3)
\hat{\varphi}(k_4) e^{i(k_2 + k_3 + k_4) x} \diff{k_2} \diff{k_3} \diff{k_4},
\label{defnonlin}
\end{equation}
where
\begin{align}
\begin{split}
T(k_2, k_3, k_4) & = -\Re \int_\R \frac{G'(\eta)}{\eta} [(1 - e^{i k_2 \eta})(1 - e^{i k_3 \eta})(1 - e^{i k_4 \eta})] \diff{\eta}
\\
& = -\Re \int_\R \frac{G'(\eta)}{\eta} \bigg\{\left(1 - e^{i k_2 \eta}\right) + \left(1 - e^{i k_3 \eta}\right)
+ \left(1 - e^{i k_4 \eta}\right) + \left(1 - e^{i (k_2 + k_3 + k_4) \eta}\right)\\
& \qquad\qquad\qquad\qquad- \left(1 - e^{i (k_2 + k_3) \eta}\right) - \left(1 - e^{i (k_2 + k_4) \eta}\right)
- \left(1 - e^{i (k_3 + k_4) \eta}\right)\bigg\} \diff{\eta}.
\end{split}
\label{defTint}
\end{align}
We assume that $\eta^2 G'(\eta)$ is integrable at $0$ and $G'(\eta)/\eta$ is integrable at infinity, as is the case for the
Green's function \eqref{defG}, and  consider three cases, depending on whether:
(a) $\eta G'(\eta)$ is integrable at $0$  ($1<\alpha\le 2$);
(b) $\eta G'(\eta)$ is nonintegrable at $0$ and  integrable at infinity ($0<\alpha<1$);
(c) $\eta G'(\eta)$ is nonintegrable at $0$ and nonintegrable at infinity ($\alpha =1$).

In case (a), we write $T$ in \eqref{defTint} as
\begin{equation}
T(k_2, k_3, k_4) = a(k_2) + a(k_3) + a(k_4) + a(k_2 + k_3 + k_4) - a(k_2 + k_3) - a(k_2 + k_4) - a(k_3 + k_4),
\label{defT}
\end{equation}
where $a \colon \R \to \R$ is defined by
\begin{equation}
a(k) = -\int_\R \frac{G'(\eta)}{\eta}\left[1 - \cos(k \eta)\right] \diff{\eta}.
\label{defaa}
\end{equation}
In case (b), we use the cancelation
\begin{equation}
k_2^2 + k_3^2 + k_4^2 + (k_2 + k_3 + k_4)^2 - (k_2 + k_3)^2 - (k_2 + k_4)^2 - (k_3 + k_4)^2 = 0
\label{kcancel}
\end{equation}
in \eqref{defTint} and write $T$ as \eqref{defT} where
\begin{equation}
a(k) = -\int_\R \frac{G'(\eta)}{\eta}\left[1 - \frac{1}{2} (k\eta)^2 - \cos(k \eta)\right] \diff{\eta}.
\label{defab}
\end{equation}
In case (c),  we use this cancelation only for $|\eta| < 1$ and write $T$ as \eqref{defT} where
\begin{equation}
a(k) = -\int_{|\eta|< 1} \frac{G'(\eta)}{\eta}\left[1 - \frac{1}{2} (k\eta)^2 - \cos(k \eta)\right] \diff{\eta}
-  \int_{|\eta|> 1} \frac{G'(\eta)}{\eta}\left[1 - \cos(k \eta)\right] \diff{\eta}.
\label{defac}
\end{equation}

In order to give a symmetric expression for $T$, it is convenient to introduce another variable $k_1$
and define $S \colon \R^4 \to \R$ by
\begin{align}
\begin{split}
S(k_1, k_2, k_3, k_4)
& = a(k_1) + a(k_2) + a(k_3) + a(k_4)
\\
&- \frac 12 \bigg\{a(k_1 + k_2) + a(k_1 + k_3) + a(k_1 + k_4)
+ a(k_2 + k_3) + a(k_2 + k_4) + a(k_3 + k_4)\bigg\}.
\end{split}
\label{defS}
\end{align}
Then
\begin{equation}
T(k_2, k_3, k_4) = S(k_1, k_2, k_3, k_4)
\qquad
\text{on $k_1 + k_2 + k_3 + k_4 = 0$}.
\label{defST}
\end{equation}
Using \eqref{defnonlin} and \eqref{defST} in \eqref{consapprox}, we see that the spectral form of  \eqref{consapprox} is
\begin{align}\label{specapprox}
\begin{split}
&\hat{\varphi}_t(k_1, t)
 + \frac 16 ik_1 \int_{\R^3} \delta(k_1 + k_2 + k_3 + k_4) S(k_1, k_2, k_3, k_4)
\hat{\varphi}^*(k_2, t) \hat{\varphi}^*(k_3, t) \hat{\varphi}^*(k_4, t) \diff{k_2}\diff{k_3}\diff{k_4}
\\
&\qquad\qquad\qquad\qquad\qquad + i k_1 b(k_1) \hat{\varphi}(k_1, t)= 0,
\end{split}
\end{align}
where $\delta$ denotes the delta-distribution, $\hat{\varphi}^*(k) = \varphi(-k)$ denotes the complex conjugate of $\hat{\varphi}(k)$,
and $b(k)$ is the symbol of $\L$.

Using  the convolution theorem and \eqref{defS} to take the inverse Fourier transform of \eqref{specapprox},
we find that the approximate equation \eqref{consapprox} can be written as
\begin{equation}\label{realapprox}
\varphi_t + \frac 12 \partial_x \bigg\{\varphi^2 \A \varphi - \varphi \A \varphi^2
+ \frac 13 \A \varphi^3\bigg\} + \L \varphi_x = 0,
\end{equation}
where $\A$ is the self-adjoint operator with symbol $a(k)$.

\subsection{Approximate equations}

In this section, we write out the explicit form of the approximate equation derived above for Euler, SQG, and gSQG fronts.

\subsubsection{Euler equation $(\alpha = 2)$}
\label{approx:euler}

For the Euler equation, we have
\[
G'(\eta) = -\frac{1}{2 \pi \eta},\qquad \L = \frac{1}{2}\modDf^{-1},
\]
and the approximate equation \eqref{consapprox} for vorticity fronts is
\[\varphi_t(x, t) + \frac{1}{12 \pi} \partial_x \int_\R \frac{[\varphi(x, t) - \varphi(x + \eta, t)]^3}{\abs{\eta}^2} \diff{\eta}
= \frac 12 \hilbert \varphi(x, t),\]
where $\hilbert$ is the Hilbert transform.
The symbol $a$ is given by \eqref{defaa}, so
\[a(k) = \frac{1}{\pi}\int_0^\infty \frac{1 - \cos(k \eta)}{\eta^2} \diff{\eta} =  \frac 12 \abs{k},\]
with the corresponding operator
\[\A = \frac 12 \abs{\partial_x}.\]
Thus, from  \eqref{realapprox}, the approximate Euler equation is
\begin{equation}
\varphi_t + \frac 14 \partial_x \bigg\{\varphi^2 \abs{\partial_x} \varphi - \varphi \abs{\partial_x} \varphi^2
+ \frac 13 \abs{\partial_x} \varphi^3\bigg\} = \frac 12 \hilbert \varphi.\
\label{approx2}
\end{equation}
This equation agrees with the asymptotic equation derived in \cite{BiHu} directly from the incompressible Euler equations
when the vorticity jumps from $-1/2$ to $1/2$ across the front.

\subsubsection{SQG equation $(\alpha = 1)$}
\label{subsec:sqg}

For the SQG equation, we have
\[
G'(\eta) = -\frac{1}{\eta |\eta|},\qquad \L = -2\log\modDf,
\]
and the approximate equation \eqref{consapprox} for SQG fronts is
\[\varphi_t(x, t) + \frac 16 \partial_x \int_\R \bigg[\frac{\varphi(x, t) - \varphi(x + \eta, t)}{\abs{\eta}}\bigg]^3 \diff{\eta}
= 2 \log\abs{\partial_x} \varphi_x(x, t).\]
The symbol $a$ is given by \eqref{defac}, so
\begin{align*}
a(k)
& = 2\int_0^1 \frac{1 - \frac 12 k^2 \eta^2 - \cos(k \eta)}{\eta^3} \diff{\eta}
+ 2\int_1^\infty \frac{1 - \cos(k \eta)}{\eta^3} \diff{\eta}
\\
&= 2k^2 \left(\int_0^{|k|} \frac{1 - \frac 12 \eta^2 - \cos \eta}{\eta^3} \diff{\eta}
+ \int_{|k|}^\infty \frac{1 - \cos \eta}{\eta^3} \diff{\eta}\right).
\end{align*}
Writing
\begin{align*}
\int_0^{|k|} \frac{1 - \frac 12 \eta^2 - \cos \eta}{\eta^3} \diff{\eta}
+ \int_{|k|}^\infty \frac{1 - \cos \eta}{\eta^3} \diff{\eta}
&=
\frac{1}{2}C -\frac{1}{2} \int_1^{|k|} \frac{\diff{\eta}}{\eta}
\end{align*}
where
\[
C = 2\int_0^1 \frac{1 - \frac 12 \eta^2 - \cos\eta}{\eta^3} \diff{\eta}
+ 2\int_{1}^{\infty} \frac{1 - \cos\eta}{\eta^3} \diff{\eta},
\]
is a  constant, we get that $a(k) = Ck^2-k^2 \log|k|$.
The term $Ck^2$ cancels out of the expression in \eqref{defS} for $S(k_1, k_2, k_3, k_4)$
on $k_1 + k_2 + k_3 + k_4 = 0$, so we can take
\begin{equation}
a(k) = -k^2 \log|k|,
\label{defsqga}
\end{equation}
which gives the kernel
\begin{equation}\label{Skernelsqg}\begin{aligned}
S(k_1, k_2, k_3, k_4)
& = -k_1^2 \log\abs{k_1} - k_2^2 \log\abs{k_2} - k_3^2 \log\abs{k_3} - k_4^2 \log\abs{k_4}\\
& \quad + \frac 12 \bigg\{(k_1 + k_2)^2 \log\abs{k_1 + k_2} + (k_1 + k_3)^2 \log\abs{k_1 + k_3} + (k_1 + k_4)^2 \log\abs{k_1 + k_4}\\
& \qquad + (k_2 + k_3)^2 \log\abs{k_2 + k_3} + (k_2 + k_4)^2 \log\abs{k_2 + k_4} + (k_3 + k_4)^2 \log\abs{k_3 + k_4}\bigg\}.
\end{aligned}\end{equation}
As a result of the cancelation \eqref{kcancel}, this function is homogeneous of degree $2$ on $k_1+k_2+k_3+k_4 = 0$.
The operator corresponding to \eqref{defsqga} is
\[
\A = \partial_x^2 \log\abs{\partial_x}.
\]
Thus, from  \eqref{realapprox}, the approximate SQG equation is
\begin{equation}\label{realapproxsqg}
\varphi_t + \frac 12 \partial_x \bigg\{\varphi^2 \log\abs{\partial_x} \varphi_{xx}
- \varphi \log\abs{\partial_x} (\varphi^2)_{xx} + \frac 13 \log\abs{\partial_x} (\varphi^3)_{xx}\bigg\}
= 2 \log\abs{\partial_x} \varphi_x.
\end{equation}
We remark that since the nonlinear term in \eqref{realapproxsqg} is scale-invariant, this approximate
equation has the same anomalous scale-invariance \eqref{sqg_scaling} as the full SQG front equation.

\subsubsection{gSQG equation}
\label{approx:gsqg}

For the gSQG equation with $0<\alpha<1$ or $1 < \alpha < 2$, we have
\[G'(\eta) = -\frac{2 - \alpha}{\eta |\eta|^{2 - \alpha}},\qquad
\L = b_\alpha \modDf^{1-\alpha},
\]
and the approximate equation \eqref{consapprox} for gSQG fronts is
\[\varphi_t(x, t)
+ \frac{1}{6}(2-\alpha)\partial_x \int_\R \frac{[\varphi(x, t) - \varphi(x + \eta, t)]^3}{\abs{\eta}^{4 - \alpha}} \diff{\eta}
+ b_\alpha \abs{\partial_x}^{1 - \alpha} \varphi_x(x, t)
= 0.\]
If $1<\alpha < 2$, then the symbol $a$ is given by \eqref{defaa}, so
\begin{align}
a(k) & = 2(2 - \alpha) \int_0^\infty \frac{1 - \cos(k \eta)}{\abs{\eta}^{4 - \alpha}} \diff{\eta}
 = c_\alpha \abs{k}^{3 - \alpha},
 \label{sqgalphaa}
\end{align}
where
\begin{align}
\begin{split}
c_\alpha &=2(2 - \alpha) \int_0^\infty \frac{1 - \cos\eta}{\eta^{4 - \alpha}} \diff{\eta}
\\
&= 4(2-\alpha)\int_0^\infty \frac{\sin^2(\eta/2)}{\eta^{4 - \alpha}} \diff{\eta}
\\
&= 2(2-\alpha) \sin\left(\frac{\pi\alpha}{2}\right) \Gamma(\alpha-3)
\\
&= \frac{b_\alpha}{3-\alpha},
\end{split}
\label{defC12}
\end{align}
with $b_\alpha$ defined in \eqref{a-const}. The kernel $S$ in \eqref{defS} is given by
\begin{equation}\label{Skernel1a2}\begin{aligned}
S(k_1, k_2, k_3, k_4) & = c_\alpha\left\{\abs{k_1}^{3 - \alpha} + \abs{k_2}^{3 - \alpha} + \abs{k_3}^{3 - \alpha} + \abs{k_4}^{3 - \alpha}\right\}\\
& \quad - \frac 12 c_\alpha\left\{\abs{k_1 + k_2}^{3 - \alpha} + \abs{k_1 + k_3}^{3 - \alpha} + \abs{k_1 + k_4}^{3 - \alpha}\right.\\
& \qquad \left.+ \abs{k_2 + k_3}^{3 - \alpha} + \abs{k_2 + k_4}^{3 - \alpha} + \abs{k_3 + k_4}^{3 - \alpha}\right\}.
\end{aligned}\end{equation}
This function is homogeneous of degree $3-\alpha$. The corresponding operator is
\[\A =c_\alpha \abs{\partial_x}^{3 - \alpha}.\]
Thus, from  \eqref{realapprox}, the approximate SQG equation is
\begin{equation}\label{realapprox1a2}
\varphi_t + \frac {1}{2}c_\alpha \partial_x \bigg\{\varphi^2 \abs{\partial_x}^{3 - \alpha} \varphi
- \varphi \abs{\partial_x}^{3 - \alpha} (\varphi^2) + \frac 13 \abs{\partial_x}^{3 - \alpha} (\varphi^3)\bigg\}
 + b_\alpha \abs{\partial_x}^{1 - \alpha} \varphi_x= 0.
\end{equation}

If $0 < \alpha < 1$, then $a$ is given by \eqref{defab}
instead of \eqref{defaa}, and we get \eqref{realapprox1a2}
with
\begin{equation}
c_\alpha = 2(2 - \alpha)  \int_0^\infty \frac{1 - \frac{1}{2}\eta^2-\cos\eta}{\eta^{4 - \alpha}} \diff{\eta}.
\label{defC01}
\end{equation}

\subsection{Hamiltonian structure}

The approximate equation has the Hamiltonian form
\[\varphi_t + \partial_x \left[\frac{\delta \H}{\delta \varphi}\right] = 0,\]
where, suppressing the time variable, we can write the Hamiltonian in  equivalent forms as
\begin{align*}
\H(\varphi) &= -\frac {1}{6 \cdot 8} \int_{\R^2}\left[\frac{G'(x - x')}{x - x'}\right]
 [\varphi(x) - \varphi(x')]^4 \diff{x}\diff{x'} + \frac 12 \int_\R \varphi(x) \L \varphi(x) \diff{x}
\\
&= \int_\R \bigg[\frac 16 \varphi \A \varphi^3 - \frac 18 \varphi^2 \A \varphi^2\bigg]\diff{x}
+ \frac 12 \int_\R \varphi \L \varphi \diff{x}.
\end{align*}
The spectral form of the Hamiltonian is
\begin{align*}
\H(\hat{\varphi}) &=-\frac {1}{6 \cdot 8} \int_{\R^4} \delta(k_1+k_2+k_3+k_4)
\\
&\qquad\qquad\qquad S(k_1,k_2,k_3,k_4) \hat{\varphi}(k_1)\hat{\varphi}(k_2)
\hat{\varphi}(k_3)\hat{\varphi}(k_4) \diff{k}_1 \diff{k}_2 \diff{k}_3 \diff{k}_4
 \\
&\quad + \frac{1}{2} \int_{\R} b(k) \hat{\varphi}^*(k)\hat{\varphi}(k) \diff{k}.
\end{align*}
This Hamiltonian structure explains the symmetry of $S$ in \eqref{defS}.

For $\alpha\ne 1$, the quadratic term in the Hamiltonian
is proportional to
\[
\int_\R \varphi(x) \modDf^{1-\alpha} \varphi(x) \diff{x},
\]
which controls the homogeneous $\dot{H}^s(\R)$-norm of $\varphi$ with
$s= {(1-\alpha)}/{2}$. The quartic term is proportional to
\[
\int_{\R^2} \frac{\left[\varphi(x)-\varphi(x')\right]^4}{|x-x'|^{4-\alpha}} \diff{x} \diff{x'},
\]
which controls the homogeneous $\dot{W}^{r,4}(\R)$-Slobodeckij norm \cite{taheri} of $\varphi$ with
$r= (3-\alpha)/{4}$.
For $0<\alpha \le 2$, we have $-1/2\le s < 1/2$, $1/4\le r <3/4$, so these norms appear too weak to be useful for well-posedness results.

\section{Local well-posedness for the approximate equation}
\label{Sec-LWP}

In this section, we study the local well-posedness of the initial value problem for the approximate
gSQG front equation with $1<\alpha \le 2$ in \eqref{realapprox1a2} and the approximate
SQG front equation in \eqref{realapproxsqg}. For simplicity,
we consider spatially periodic functions with zero mean.
The analysis for the SQG equation is more delicate than for the gSQG equation, and we
obtain a weaker result in that case, in which solutions may lose Sobolev derivatives over time.

The nonlinear fluxes in these equations appear to involve derivatives, but this is misleading because of a cancelation,
as the the estimates below will show. For smooth solutions, a cartoon of the gSQG equation \eqref{realapprox1a2} with  $1<\alpha \le 2$ is a cubically nonlinear conservation law with a lower-order dispersive term of
order less than one,
\[
\varphi_t + \left(\varphi^3\right)_x + \modDf^{1-\alpha} \varphi_x=0.
\]
Additional logarithmic derivatives arise for the SQG equation \eqref{realapproxsqg}, and for smooth, spatially periodic solutions a rough cartoon of the equation is
\[
\varphi_t + \left(3\varphi^2\Lop \varphi - 4\varphi\Lop \varphi^2 + \Lop\varphi^3\right)_x = \Lop\varphi_x,
\qquad \Lop = \log\modDf.
\]

\subsection{Notation} We denote the Fourier coefficients of a $2\pi$-periodic function (or distribution) $f \colon \T\to \R$ by
\[
\hat{f}(k) = \frac{1}{2\pi} \int_{\T} f(x) e^{-ikx} \diff{x}
\]
and the $\ell^p$-norm of $\hat{f} \colon \Z_* \to \C$ by
\[
\|\hat{f}\|_{\ell^p} = \left(\sum_{k\in\Z_*} |\hat{f}(k)|^p\right)^{1/p},
\]
where $\Z_* = \Z \setminus \{0\}$ the set of nonzero integers.

For $s\in \R$, we let
\[
\dot{H}^s(\T) = \left\{f \colon \T \to \R \mid \text{$\hat{f}(0) = 0$,  $\norm{f}_{\dot{H}^s} < \infty$}\right\}
\]
denote the Hilbert space of zero-mean, periodic functions with square-integrable derivatives of the order $s$, and norm
\begin{align}
\norm{f}_{\dot{H}^s}
&=\left(\sum_{k \in \Z_*} \abs{k}^{2s} |\hat{f}(k)|^2\right)^{1/2}.
\label{Hsnorm}
\end{align}

We will use the following consequence of Young's inequality
\begin{equation}
\sum_{\substack{k_1, k_2, k_3, k_4 \in \Z_*\\ k_1 + k_2 + k_3 + k_4 = 0}}
\left| \hat{f}_1(k_1) \hat{f}_2(k_2) \hat{f}_3(k_3) \hat{f}_4(k_4)\right| \le
\|\hat{f}_1\|_{\ell^2} \|\hat{f}_2\|_{\ell^1}
\|\hat{f}_3\|_{\ell^1} \|\hat{f}_4\|_{\ell^2}
\label{convest}
\end{equation}
and the Sobolev inequality
\begin{equation}
\|\hat{f}\|_{\ell^1} \le  Z(s) \|f\|_{\dot{H}^s}
\qquad \text{for $s > 1/2$},
\label{sobest}
\end{equation}
where $Z$ is given in terms of the Riemann-zeta function by
\[
Z(s) = \left(\sum_{k\in \Z_*} \frac{1}{|k|^{2s}}\right)^{1/2}
= \sqrt{2\zeta(2s)}.
\]

Let $\rho \colon \Z_*^4 \to \Z_*^4$ be a map that permutes its entries and orders their
absolute values. We denote the values of $\rho$ by $(m_1,m_2,m_3,m_4) = \rho(k_1,k_2,k_3,k_4) $,
where
\begin{align}
&(m_1,m_2,m_3,m_4) = (k_{\sigma1},k_{\sigma2},k_{\sigma3},k_{\sigma4})
\quad \text{for some $\sigma\in S_4$},
\label{defkm2}
\\
&\abs{m_1} \geq \abs{m_2} \geq \abs{m_3} \geq \abs{m_4}.
\label{defkm1}
\end{align}
Here, $S_4$ denotes the symmetric group on $\{1, 2, 3, 4\}$.

\subsection{Local well-posedness for the approximate gSQG equation $(1 < \alpha \le 2)$}

In this section, we prove short-time existence and uniqueness for spatially periodic solutions of the initial value
problem for the gSQG equation,
\begin{align}\label{gSQGivp}
\begin{split}
&\varphi_t + \frac {1}{2}c_\alpha \partial_x \bigg\{\varphi^2 \abs{\partial_x}^{3 - \alpha} \varphi
- \varphi \abs{\partial_x}^{3 - \alpha} (\varphi^2) + \frac 13 \abs{\partial_x}^{3 - \alpha} (\varphi^3)\bigg\}
 + b_\alpha \abs{\partial_x}^{1 - \alpha} \varphi_x= 0.
\\
&\varphi(x,0) = \varphi_0(x).
\end{split}
\end{align}
We begin with a general result that is the analog for cubically nonlinear equations of the well-posedness result in  \cite{Hun}
for quadratically nonlinear equations. The proof depends crucially on the symmetry of the interaction coefficients that follows from the Hamiltonian structure of the equation.

Consider the spectral form of an initial value problem for a spatially-periodic function
${\varphi}(x,t)$, with Fourier coefficients $\hat{\varphi}(k,t)$, given by
\begin{align}\label{speceqn}
\begin{split}
& \hat{\varphi}_t(k_1, t)
+ \frac 16 ik_1\sum_{\substack{k_2, k_3, k_4 \in \Z_*\\ k_2 + k_3 + k_4 = -k_1}}S(k_1, k_2, k_3, k_4)
 \hat{\varphi}^*(k_2, t) \hat{\varphi}^*(k_3, t) \hat{\varphi}^*(k_4, t)
 + ik_1b(k_1) \hat{\varphi}(k_1,t) = 0,
\\
&\hat{\varphi}(k_1,0) = \hat{\varphi}_0(k_1).
\end{split}
\end{align}
When convenient, we omit the time variable and write $\hat{\varphi}(k) = \hat{\varphi}(k, t) = \hat{\varphi}^*(-k)$,
$\varphi_j = \varphi(k_j)$.
In \eqref{speceqn}, we
assume that $S \colon \Z_*^4 \to \R$ satisfies
\begin{align}
S(k_1, k_2, k_3, k_4) &= S(-k_1, -k_2, -k_3, -k_4),\label{Scond1}
\\
S(k_1, k_2, k_3, k_4) &= S(k_{\sigma 1}, k_{\sigma 2}, k_{\sigma 3}, k_{\sigma 4})
\quad \text{for every $\sigma\in S_4$}\label{Scond2}
\end{align}
and  that there exist $\mu,\nu \ge 0$ such that
\begin{align}
& \abs{S(k_1, k_2, k_3, k_4)}
\leq C_S\abs{m_3}^{\mu} \abs{m_4}^\nu
\qquad \text{for all $k_1,k_2,k_3,k_4\in \Z_*$},
\label{Scond3}
\end{align}
where $m_1,m_2,m_3,m_4$ are defined as in \eqref{defkm2}--\eqref{defkm1},
and $C_S$ is a constant. That is, the growth of $S$ is bounded by the smaller wavenumbers on which it depends.

\begin{theorem}\label{lwp-gsqg1a2}
Suppose that $S \colon \Z_*^4 \to \R$ satisfies \eqref{Scond1}--\eqref{Scond3},
and $b \colon \Z_* \to \R$ is a bounded, even function. If
\[
s > \max\left\{\mu + \frac{3}{2}, \nu + \frac{1}{2}\right\},
\]
then for every $\varphi_0 \in \dot{H}^s(\T)$, there exists $T > 0$, depending on $\|\varphi_0\|_{\dot{H}^s}$,
such that the initial value problem \eqref{speceqn} has a solution
\[
\varphi \in C\left([0, T]; \dot{H}^s(\T)\right) \cap C^1\left([0, T]; \dot{H}^{s - 1}(\T)\right).
\]
Furthermore, the solution is unique if
\begin{equation}
s > \max\left\{\mu + \frac{3}{2}, \nu + \frac{3}{2}\right\}.
\label{smnineq1}
\end{equation}
\end{theorem}

\begin{proof}
We prove the main \emph{a priori} estimates and only sketch the proof, which follows by standard arguments for quasilinear hyperbolic PDEs \cite{Ta}.

Multiplying \eqref{speceqn} by $\hat{\varphi}^*(k_1)$, taking the real part, and using \eqref{Scond2}
to symmetrize the result, we get that
\begin{align}
\begin{split}
\frac{\diff}{\diff{t}} \norm{\varphi}_{\dot{H}^s}^2
&\leq
\frac{1}{12}\sum_{\substack{k_1, k_2, k_3, k_4 \in \Z_*\\ k_1 + k_2 + k_3 + k_4 = 0}}
\biggl|\left(k_1 \abs{k_1}^{2s} + k_2 \abs{k_2}^{2s} + k_3 \abs{k_3}^{2s} + k_4 \abs{k_4}^{2s}\right)\biggr.
\\
&\qquad\qquad\qquad\qquad\qquad\quad \biggl.S(k_1, k_2, k_3, k_4)
\hat{\varphi}(k_1) \hat{\varphi}(k_2) \hat{\varphi}(k_3) \hat{\varphi}(k_4)\biggr|.
\end{split}
\label{energyeqn}
\end{align}
Using Lemma~\ref{2s+1ineq}, the permutation property \eqref{defkm2} of the $m_j$, the symmetry of $S$ in \eqref{Scond2},
and the estimate \eqref{Scond3}, we get that
\begin{align}
\begin{split}
\frac{\diff}{\diff{t}} \norm{\varphi}_{\dot{H}^s}^2 &\leq
\frac{1}{12} C_0(s)\sum_{\substack{k_1, k_2, k_3, k_4 \in \Z_*\\ k_1 + k_2 + k_3 + k_4 = 0}}
\abs{m_1}^s \abs{m_2}^s \abs{m_3} \cdot
\abs{S(k_1, k_2, k_3, k_4) \hat{\varphi}(k_1) \hat{\varphi}(k_2) \hat{\varphi}(k_3) \hat{\varphi}(k_4)}
\\
&\leq
\frac{1}{12} C_S C_0(s)\sum_{\substack{k_1, k_2, k_3, k_4 \in \Z_*\\ k_1 + k_2 + k_3 + k_4 = 0}}\abs{m_1}^s \abs{m_2}^s \abs{m_3}^{\mu + 1} \abs{m_4}^\nu
\abs{ \hat{\varphi}(m_1) \hat{\varphi}(m_2) \hat{\varphi}(m_3) \hat{\varphi}(m_4)}.
\end{split}
\label{tempest}
\end{align}

For fixed $(k_1,k_2,k_3,k_4)\in \Z_*^4$ with corresponding
$(m_1,m_2,m_3,m_4)\in \Z_*^4$, as in \eqref{defkm2}--\eqref{defkm1}, we have
\begin{align*}
&\abs{m_1}^s \abs{m_2}^s \abs{m_3}^{\mu + 1} \abs{m_4}^\nu
\abs{ \hat{\varphi}(m_1) \hat{\varphi}(m_2) \hat{\varphi}(m_3) \hat{\varphi}(m_4)}
\\
&\qquad\qquad \le
\sum_{\substack{(k_1', k_2', k_3', k_4')  =\\ (k_{\sigma1}, k_{\sigma2}, k_{\sigma3}, k_{\sigma4}),\ \sigma\in S_4 }}
\abs{k_1'}^s \abs{k_2'}^s \abs{k_3'}^{\mu + 1} \abs{k_4'}^\nu
\abs{ \hat{\varphi}(k_1') \hat{\varphi}(k_2') \hat{\varphi}(k_3') \hat{\varphi}(k_4')}
\end{align*}
Using this inequality to estimate the sum of terms in \eqref{tempest}
depending on $m_j$ by a sum depending on $k_j$, followed by
the inequalities \eqref{convest}--\eqref{sobest}, we get that
\begin{align*}
\frac{\diff}{\diff{t}} \norm{\varphi}_{\dot{H}^s}^2 &\leq
\frac{4!}{12} C_S C_0(s) \sum_{\substack{k_1, k_2, k_3, k_4 \in \Z_*\\ k_1 + k_2 + k_3 + k_4 = 0}}
\abs{k_1}^s \abs{k_2}^s \abs{k_3}^{\mu + 1} \abs{k_4}^\nu
\abs{ \hat{\varphi}(k_1) \hat{\varphi}(k_2) \hat{\varphi}(k_3) \hat{\varphi}(k_4)}
\\
&\leq  2C_S C_0(s) \norm{\varphi}_{\dot{H}^s}
\|\,|k|^{\mu+1}\hat{\varphi}\|_{{\ell}^1(\Z_*)} \|\,|k|^{\nu}\hat{\varphi}\|_{{\ell}^1(\Z_*)}
\norm{\varphi}_{\dot{H}^s}
\\
&\le C_4(s) \norm{\varphi}_{\dot{H}^s}^4.
\end{align*}
Thus, Gr\"{o}nwall's inequality gives the \emph{a priori} estimate
\begin{equation*}
\norm{\varphi(t)}_{\dot{H}^s}^2 \leq \frac{1}{C_4} \left(\frac{1}
{T_* - t}\right),
\qquad
T_* = \frac{1}{C_4 \norm{\varphi_0}_{\dot{H}^s}^2},
\end{equation*}
and it follows that
\begin{equation}
\sup_{0\le t \le T} \|\varphi(t)\|_{\dot{H}^s}^2 \le \frac{1}{C_4} \left(\frac{1}
{T_* - T}\right),
\label{phibd}
\end{equation}
 for any $0<T<T_*$.

To estimate $\varphi_t$, we write \eqref{speceqn} in spatial form as
\begin{equation}
\varphi_t + \Df \tf(\varphi,\varphi,\varphi) + \L\varphi_x= 0,
\label{phiteq}
\end{equation}
where $\L$ is a bounded operator on $\dot{H}^s(\T)$, and the trilinear operator
$\tf$ is defined in terms of Fourier coefficients
by
\begin{equation}
\hat{\tf}(\hat{\varphi},\hat{\psi},\hat{\chi})(k_1) =
 \frac 16 \sum_{\substack{k_2, k_3, k_4 \in \Z_*\\ k_2 + k_3 + k_4 = -k_1}}S(k_1, k_2, k_3, k_4)
 \hat{\varphi}^*(k_2) \hat{\psi}^*(k_3) \hat{\chi}^*(k_4).
\label{defFform}
\end{equation}
The symmetry of $S$ implies that
\begin{equation*}
q(\eta, \varphi,\psi,\chi) = \int_\T \eta \tf(\varphi,\psi,\chi)\diff{x}
\end{equation*}
is a symmetric form. Moreover, using \eqref{Scond3}, we get that
\begin{align}
\begin{split}
\left|q(\eta, \varphi,\psi,\chi) \right| &\le
\frac{\pi}{3} \sum_{\substack{k_1,k_2, k_3, k_4 \in \Z_*\\ k_1+ k_2 + k_3 + k_4 = 0}}
|S(k_1, k_2, k_3, k_4) \hat{\eta}(k_1) \hat{\varphi}(k_2) \hat{\psi}(k_3) \hat{\chi}(k_4)|
\\
&\le \frac{\pi}{3}C_S\sum_{\substack{k_1,k_2, k_3, k_4 \in \Z_*\\ k_1+ k_2 + k_3 + k_4 = 0}}
|k_1|^{-s}|\hat{\eta}(k_1)|\cdot |k_1|^s |m_3|^\mu |m_4|^\nu |\hat{\varphi}(k_2) \hat{\psi}(k_3) \hat{\chi}(k_4)|.
\end{split}
\label{Fintest}
\end{align}
On $k_1+ k_2 + k_3 + k_4 = 0$, we have
\[
|k_1|^s \le Y(s) \left(|k_2|^{s} + |k_3|^{s} + |k_4|^{s}\right).
\]
From \eqref{defkm1}, we have  for any $1\le p\ne q\le 4$ that
\[
 |m_3|^\mu |m_4|^\nu \le |k_p|^\mu |k_q|^\nu +  |k_p|^\nu |k_q|^\mu.
\]
Choosing $\{p,q\}$ disjoint from $\{1,2\}$, $\{1,3\}$, or $\{1,4\}$ as appropriate, we get that
\begin{align*}
|k_1|^s |m_3|^\mu |m_4|^\nu &\le Y\bigl[
|k_2|^{s}\left(|k_3|^\mu |k_4|^\nu +|k_3|^\nu |k_4|^\mu\right)\bigr.
\\
&\qquad+
|k_3|^{s}\left(|k_2|^\mu |k_4|^\nu +|k_2|^\nu |k_4|^\mu\right)
\\
&\qquad+
\bigl.|k_4|^{s}\left(|k_2|^\mu |k_3|^\nu +|k_2|^\nu |k_3|^\mu\right)\bigr].
\end{align*}
Using this inequality in \eqref{Fintest}, followed by
\eqref{convest}--\eqref{sobest} with the assumption that
\begin{equation}
s > \max\left\{\mu+\frac{1}{2}, \nu+\frac{1}{2}\right\},
\label{sineq1}
\end{equation}
we get
\begin{align*}
\left|q(\eta, \varphi,\psi,\chi) \right|
&\le C \|\eta\|_{\dot{H}^{-s}}  \|\varphi\|_{\dot{H}^s} \|\psi\|_{\dot{H}^s} \|\chi\|_{\dot{H}^s},
\end{align*}
where $C$ denotes a constant. It follows by duality that \eqref{defFform} defines a bounded trilinear map
\begin{equation}
\tf \colon \dot{H}^s(\T)\times\dot{H}^s(\T)\times\dot{H}^s(\T)\to \dot{H}^s(\T)
\label{conF}
\end{equation}
when $s$ satisfies \eqref{sineq1}. Hence, \eqref{phibd}--\eqref{phiteq} imply that
\begin{equation}
\sup_{0\le t \le T} \|\varphi_t\|_{\dot{H}^{s - 1}} \le C.
\label{est2}
\end{equation}

Moreover, if $\varphi$, $\psi$ are solutions of \eqref{phiteq}
with initial data $\varphi(0) = \varphi_0$, $\psi(0) = \psi_0$,
then writing $u=\varphi-\psi$, using the symmetry of $q$, and
the identity
\[
q(\eta_x, \varphi,\psi,\chi)+ q(\eta, \varphi_x,\psi,\chi)+ q(\eta, \varphi,\psi_x,\chi)
+ q(\eta, \varphi,\psi,\chi_x) =0,
\]
we get that
\begin{equation}
 \frac{\diff}{\diff{t}} \|u\|_{L^2}^2
+ 2q(u,u,\varphi,\varphi_x) +  q(u,u,\varphi_x,\psi) + q(u,u,\varphi,\psi_x)+ 2q(u,u,\psi,\psi_x)=0.
\label{stabest}
\end{equation}
For $s> \max\{\mu+1/2, \nu+1/2\}$, it follows as in \eqref{Fintest} that
\[
|q(\eta,\varphi,\psi,\chi)| \le C \|\eta\|_{L^2} \|\varphi\|_{L^2} \|\psi\|_{\dot{H}^s} \|\chi\|_{\dot{H}^s}.
\]
Hence, when $s$ satisfies \eqref{smnineq1}, we have
\[
 \frac{\diff}{\diff{t}} \|u\|_{L^2}^2 \le C\left(\|\varphi\|^2_{\dot{H}^s} + \|\psi\|^2_{\dot{H}^s}\right)
 \|u\|_{L^2}^2,
\]
so Gr\"onwall's inequality gives the \emph{a priori} $L^2$-stability estimate
\begin{equation}
\sup_{0\le t\le T}\left\|\varphi(t)-\psi(t)\right\|^2_{L^2} \le \exp \left[C\int_0^T\left(\|\varphi(t)\|^2_{\dot{H}^s} + \|\psi(t)\|^2_{\dot{H}^s}\right)\diff{t}\right]
\left\|\varphi_0-\psi_0\right\|^2_{L^2}.
\label{L2stab}
\end{equation}

The result then follows by standard methods. We construct Galerkin approximations $\{\varphi^N: N\in \N\}$
by projecting the equations onto Fourier modes with $|k|\le N$. These approximations
satisfy the same estimates as the \emph{a priori} estimates derived above, so from \eqref{phibd} and \eqref{est2}
we can extract a subsequence that converges weakly to a limit $\varphi$
in $L^\infty(0,T; \dot{H}^s(\T)) \cap W^{1,\infty}(0,T; \dot{H}^{s-1}(\T))$.
By the Aubin-Lions lemma, a further subsequence converges strongly in  $C([0,T]; \dot{H}^{s-\epsilon}(\T))$
for sufficiently small $\epsilon>0$, and by the continuity of the nonlinear term in  \eqref{conF}, the limit
is a solution of the equation. The fact that $\varphi \in C([0,T]; \dot{H}^s(\T))$ follows from
 weak continuity $\varphi \in C_w([0,T]; \dot{H}^s(\T))$
and continuity of the norm $\|\varphi\|_{\dot{H}^s}$, and uniqueness follows from \eqref{L2stab} when $s$
satisfies \eqref{smnineq1}.
\end{proof}

Since \eqref{speceqn} is reversible, Theorem~\ref{lwp-gsqg1a2} also holds backward in time, and
a similar result would apply to the spatial case $\varphi(\cdot,t) \colon \R \to \R$.
One could also
prove  continuous dependence of the solution on the initial data by a Bona-Smith type argument,
but we will not carry out the details here.

\begin{theorem}
\label{th:alpha}
Suppose that $1 < \alpha \le 2$ and  $s > {7}/{2} - \alpha$. Then for every $\varphi_0 \in \dot{H}^s(\T)$,
there exists $T > 0$, depending on $\|\varphi_0\|_{\dot{H}^s}$, such that the initial value problem
\eqref{gSQGivp} has a solution with
\[
\varphi \in C\left([0, T]; \dot{H}^s(\T)\right) \cap C^1\left([0, T]; \dot{H}^{s - 1}(\T)\right).
\]
The solution is unique if $s > {5}/{2}$.
\end{theorem}

\begin{proof}
From Lemma~\ref{kernelest1a2}, the kernel $S$ for \eqref{gSQGivp}
satisfies \eqref{Scond1}--\eqref{Scond3} with $\mu = 2 - \alpha$ and $\nu = 1$, and the symbol $b$
for \eqref{gSQGivp} is bounded,
so the result follows from Theorem~\ref{lwp-gsqg1a2}.
\end{proof}

The Euler case $\alpha=2$ of this Theorem was proved previously in \cite{Ifr}.

\subsection{Weak local well-posedness for the approximate SQG equation}

Theorem~\ref{lwp-gsqg1a2} does not apply to the approximate SQG equation \eqref{realapproxsqg},
because its kernel \eqref{Skernelsqg} does not satisfy the estimate in \eqref{Scond3}.
Instead, there is an additional logarithmic factor and, in the absence of dispersion, the nonlinear term appears to lead to a loss of derivatives
at some finite rate.

In this section, we prove a weak local well-posedness theorem for the initial value problem
\begin{align}
\label{sqgivp}
\begin{split}
&\varphi_t + \frac 12 \partial_x \bigg\{\varphi^2 \log\abs{\partial_x} \varphi_{xx}
- \varphi \log\abs{\partial_x} (\varphi^2)_{xx} + \frac 13 \log\abs{\partial_x} (\varphi^3)_{xx}\bigg\}
+\Lop\varphi_x = 0,
\\
&\varphi(x,0) = \varphi_0(x),
\end{split}
\end{align}
where $\Lop$ is an arbitrary self-adjoint operator. The case $\Lop = -2 \log\abs{\partial_x}$ corresponds to the approximate SQG front equation.

Our proof is adapted from proofs for Gevrey-class solutions of nonlinear evolution
equations (see e.g., \cite{FrVi, KuTeViZi}), in which one uses time-dependent norms to compensate
for the loss of regularity. The difference here is that, since there is only a logarithmic derivative
loss, we obtain solutions for initial data with finitely many derivatives, rather than $C^\infty$ Gevrey-class initial data.
The existence time in the theorem depends on the number
of Sobolev derivatives possessed by the initial data as well as its Sobolev norm.

In addition to $\dot{H}^s(\T)$, we use a logarithmically-modified Hilbert space
\begin{align}
\begin{split}
\dot{H}^s_{\log}(\T) &= \left\{f \colon \T \to \R \mid \text{$\hat{f}(0) = 0$, $\norm{f}_{\dot{H}_{\log}^s} < \infty$}\right\},
\\
\norm{f}_{\dot{H}_{\log}^s} &= \left[\sum_{k \in \Z_*} \log(1 + |k|) \cdot \abs{k}^{2s} \abs{\hat{f}(k)}^2\right]^{1/2}.
\end{split}
\label{Hslog}
\end{align}
If $\tau \colon [0,T_*)\to [0,\infty)$ is a decreasing function, then we denote by $L^\infty(0,T_*; \dot{H}^\tau(\T))$ the space of functions
\[
\varphi \colon [0,T_*)  \to \bigcup_{t\in [0,T_*)} \dot{H}^{\tau(t)}(\T)
\]
such that $\varphi(t) \in \dot{H}^{\tau(t)}(\T)$, and for every $0<T<T_*$
\[
\varphi\in L^\infty(0,T; \dot{H}^{\tau_1}(\T))\qquad \text{$\tau_1 = \tau(T)$},
\]
with analogous notation for other time-dependent Sobolev spaces.

\begin{theorem}
\label{sqglwpthm}
Let the operator $\Lop$ have real-valued symbol $b : \Z \to \R$ and suppose that $\tau_0 > 5/2$.
For every $\varphi_0 \in \dot{H}^{\tau_0}(\T)$, there exists $T_* > 0$ and a
differentiable, decreasing function
$\tau \colon [0,T_*) \to (5/2,\tau_0]$ with $\tau(0) = \tau_0$, depending on $\tau_0$ and
$\|\varphi_0\|_{\dot{H}^{\tau_0}(\T)}$, such that
the initial value problem \eqref{sqgivp} has a solution with
\[
\varphi \in L^\infty(0, T_*; \dot{H}^{\tau}(\T)) \cap L^2(0, T_*; \dot{H}^{\tau}_{\log}(\T)).
\]
Moreover, there exists a numerical constant $C>0$ such that
\begin{equation}\label{sqgest}
\sup_{t \in [0, T]} \norm{\varphi(t)}_{\dot{H}^{\tau(t)}}^2
+ C \norm{\varphi_0}_{\dot{H}^{\tau_0}}^2\int_0^{T} \norm{\varphi(t)}_{\dot{H}^{\tau(t)}_{\log}}^2 \diff{t}
\leq \norm{\varphi_0}_{\dot{H}^{\tau_0}}^2
\end{equation}
for every $0<T<T_*$, where the norms are defined in  \eqref{Hsnorm}, \eqref{Hslog}. The solution is unique while
$\tau(t) > 9/2$.
\end{theorem}

\begin{proof}
First, we derive the \emph{a priori} estimate \eqref{sqgest}. Let
$\tau \colon [0,T]\to (5/2,\infty)$
be a differentiable
function, and let $\varphi$ be a smooth solution of \eqref{sqgivp}. We define energies
$E, F \colon [0,T] \to [0,\infty)$
by
\begin{align*}
E(t) &=  \norm{\varphi(t)}_{\dot{H}^{\tau(t)}}^2
= \sum_{k \in \Z_*} \abs{k}^{2 \tau(t)} \abs{\hat{\varphi}(k,t)}^2,
\\
F(t) &= \norm{\varphi(t)}_{\dot{H}^{\tau(t)}_{\log}}^2
= \sum_{k \in \Z_*} \log(1+\abs{k}) \cdot \abs{k}^{2\tau(t)} \abs{\hat{\varphi}(k,t)}^2.
\end{align*}

We write the equation in the spectral form \eqref{specapprox} with kernel \eqref{Skernelsqg}.
Using the energy equation \eqref{energyeqn}, Lemma~\ref{2s+1ineq}, and
Corollary~\ref{kernelestsqg-int} to estimate the time-derivative of $E$, we get that
\begin{align*}
\frac{\diff E}{\diff{t}} & = 2 \dot{\tau} \sum_{k \in \Z_*} \log\abs{k}
\cdot \abs{k}^{2\tau} \abs{\hat{\varphi}(k)}^2
+ \sum_{k \in \Z_*} \abs{k}^{2\tau} \frac{\diff}{\diff{t}} \abs{\hat{\varphi}(k)}^2
\\
& \leq 2 \dot{\tau} F + \frac{1}{12} \sum_{\substack{k_1, k_2, k_3, k_4 \in \Z_*\\ k_1 + k_2 + k_3 + k_4 = 0}}
\left|\left(k_1 \abs{k_1}^{2 \tau} + k_2 \abs{k_2}^{2 \tau} + k_3 \abs{k_3}^{2 \tau} + k_4 \abs{k_4}^{2 \tau}\right)S(k_1, k_2, k_3, k_4)
\hat{\varphi}_1 \hat{\varphi}_2 \hat{\varphi}_3 \hat{\varphi}_4\right|
\\
& \leq 2 \dot{\tau} F + \frac{4! }{12} C_0(\tau) C_2
\sum_{\substack{k_1, k_2, k_3, k_4 \in \Z_*\\ k_1 + k_2 + k_3 + k_4 = 0}}
\left[\log(1 + |k_1|) \log(1 + |k_2|)\right]^{1/2}  \abs{k_1}^\tau \abs{k_2}^\tau \abs{k_3}^2 \abs{k_4}
\cdot \abs{\hat{\varphi}_1 \hat{\varphi}_2 \hat{\varphi}_3 \hat{\varphi}_4}
\\
& \leq 2 \dot{\tau} F + 2C_0(\tau) C_2F
\cdot \left(\sum_{k_3 \in \Z_*} \abs{k_3}^2 \abs{\hat{\varphi}(k_3)}\right)
\cdot \sum_{k_4 \in \Z_*} \abs{k_4} \abs{\hat{\varphi}(k_4)},
\end{align*}
where a dot denotes a time derivative.
Then, as long as $\tau > 5 / 2$, the Sobolev inequality \eqref{sobest} implies that
\begin{equation}
\label{sqgenergyest1}
\frac{\diff E}{\diff{t}} \leq 2 \left[\dot{\tau} + C_3(\tau) E\right] F,\qquad
C_3(s) = C_0(s) C_2 Z(s-1) Z(s-2).
\end{equation}
The function $C_3 \colon (5/2,\infty) \to (0,\infty)$ is a smooth function such that $C_3(s) \to \infty$ as  $s\to 5/2$
and $s\to \infty$. Thus, there is a numerical constant $C_4>0$ such that
\[
C_3(s) \ge C_4\qquad \text{for $5/2<s<\infty$}.
\]
For example, if $C_0$, $C_2$ are  given by \eqref{numC0}, \eqref{numC2}, then we find numerically that
one can take $C_4=1000$.

Fix a constant  $M > 1$ and let $\tau$ be the solution of the initial value problem
\begin{equation}
\dot{\tau} + M E_0 C_3 (\tau) = 0,\qquad \tau(0) = \tau_0
\label{taueq}
\end{equation}
on a maximal time-interval $[0,T_*)$ such that $\tau(t)>5/2$, where $E_0 = E(0)$. Then
it follows from \eqref{sqgenergyest1}--\eqref{taueq} that $E$ is decreasing on $[0,T_*)$ and
\begin{equation*}
\frac{\diff E}{\diff{t}}  + (M-1) C_3 E_0 F \le 0.
\end{equation*}
Gr\"onwall's inequality gives
\[
E(t) + E_0 (M-1)  \int_0^t C_3(\tau(s)) F(s) \diff{s} \leq E_0,
\]
so \eqref{sqgest} follows for $0\le T<T_*$ with $C = (M-1) C_4$.

We define a trilinear form $\tf$ by \eqref{defFform} where $S$ is given by  \eqref{Skernelsqg}. By a similar argument to the one
in the proof of Theorem~\ref{lwp-gsqg1a2}, using Corollary~\ref{kernelestsqg-int}, we see that
$\tf \colon \dot{H}^s(\T) \times \dot{H}^s(\T) \times \dot{H}^s(\T)\to \dot{H}^s(\T)$ is bounded for $s>1/2$.
It follows from the equation for $\varphi$ and \eqref{sqgest} that if $0<T<T_*$, then
\[
\sup_{0\le t \le T} \|\varphi_t(t)\|_{\dot{H}^{\tau_1-1}} \le C
\]
where $\tau_1 = \tau(T)$, for some constant $C$ depending on $\tau_0$, T, and $E_0$.

The construction of the solution by the use of Galerkin approximations follows by standard arguments, as in the proof of Theorem~\ref{lwp-gsqg1a2},
and we omit the details.

Finally, if $\varphi$, $\psi$ are solutions \eqref{sqgivp} with initial data $\varphi(0) = \varphi_0$, $\psi(0) = \psi_0$,
then we let $\tau$ be the solution of \eqref{taueq} with
$E_0=\max\{\|\varphi_0\|^2_{H^{\tau_0}}, \|\psi_0\|^2_{H^{\tau_0}}\}$,
and we define
\[
U(t) = \|\varphi(t)-\psi(t)\|_{H^{\tau(t)-2}},\qquad V(t) =  \|\varphi(t)-\psi(t)\|_{H_{\log}^{\tau(t)-2}},
\]
where we assume that ${\tau(t)-2} > 5/2$. Then
a similar argument to the derivation of the energy estimate \eqref{sqgenergyest1} and the stability estimate \eqref{stabest},
whose details we omit, gives that
\begin{align*}
\frac{dU}{dt} &\le  \left[2 \dot{\tau} + E_0 C(\tau)\right]  V
+ C(\tau)\left(\|\varphi\|_{H^{\tau}_{\log}}+\|\psi\|_{H^{\tau}_{\log}}\right) U,
\end{align*}
where $C(\tau)>0$ is a continuous function of $\tau$. If $0<T<T_*$, then \eqref{taueq} implies that $\tau(t)$ is bounded
independently of $M$ on a time-interval $0\le t \le T/M$. We choose $M$ large enough that $M C_3(\tau) \ge C(\tau)$
on this interval. Then
\[
\frac{dU}{dt} \le   C(\tau)\left(\|\varphi\|_{H^{\tau}_{\log}}+\|\psi\|_{H^{\tau}_{\log}}\right) U
\]
for $0\le t \le T/M$, and Gr\"onwall's inequality implies that the solution is unique.
\end{proof}

\section{Traveling waves and the NLS equation}
\label{sec:nls}

We look for periodic, zero-mean traveling wave solutions of \eqref{sqg_eq} of the form
\[
\varphi = \varphi(kx-\omega t),\qquad  \varphi(\theta + 2\pi) = \varphi(\theta).
\]
These traveling waves satisfy
\begin{align*}
&k \Lop \varphi - \omega \varphi
+ \frac{1}{2} k \left\{\varphi^2 \Aop\varphi - \varphi\Aop\varphi^2
+\frac{1}{3}\Aop\varphi^3\right\} = 0,
\end{align*}
where
\begin{align*}
\Aop e^{in\theta} &= a(nk) e^{in\theta},\qquad  \Lop e^{in\theta} = b(nk) e^{in\theta},
\\
a(k) &= \begin{cases}
\frac{1}{2}|k| & \text{if $\alpha = 2$ (Euler)},
\\
c_\alpha |k|^{3-\alpha} & \text{if $0<\alpha < 1$ or $1 < \alpha < 2$},
\\
-k^2\log|k|  & \text{if $\alpha = 1$ (SQG)},
\end{cases}
\\
b(k) &=  \begin{cases}
\frac{1}{2}|k|^{-1} & \text{if $\alpha = 2$ (Euler)},
\\
b_\alpha |k|^{1-\alpha} & \text{if $0<\alpha < 1$ or $1 < \alpha < 2$},
\\
-2 \log |k| & \text{if $\alpha = 1$ (SQG)}.
\end{cases}
\end{align*}
The existence of an analytic branch of small-amplitude traveling waves follows from
the Crandall-Rabinowitz theorem for bifurcation from a simple eigenvalue \cite{Ze}.

A Fourier expansion for small-amplitude solutions of the form
\begin{align*}
\varphi(\theta;\epsilon) = \sum_{n=0}^\infty \epsilon^{2n+1} \psi_{2n+1} e^{i(2n+1)\theta} + \text{c.c.},
\qquad
\omega(\epsilon) &= \sum_{n=0}^\infty \epsilon^{2n} \omega_{2n}
\end{align*}
gives
\begin{align}
\begin{split}
\omega_0 &= k b(k) = \begin{cases}
\frac{1}{2}\sgn k & \text{if $\alpha = 2$ (Euler)},
\\
b_\alpha k|k|^{1-\alpha} & \text{if $0<\alpha < 1$ or $1 < \alpha < 2$},
\\
-2 k\log |k| & \text{if $\alpha = 1$ (SQG)},
\end{cases}
\\
\omega_2 &= \sigma_2|\psi_1|^2,
\qquad
\sigma_2 = \frac{1}{2}k \left[4a(k) - a(2k)\right].
\end{split}
\label{defw0}
\end{align}
In addition, one finds that
\[
\psi_3 = \frac{1}{2}\left[ \frac{a(k) - a(2k) + \frac{1}{3}a(3k)}{b(k) - b(3k)}\right] \psi_1^3.
\]

We remark that in the case of the approximate equation \eqref{approx2} for Euler,
with $\alpha = 2$ and $a(k) = |k|/2$, we get that
\[
a(k) - a(2k) + \frac{1}{3}a(3k) = 0,
\]
so $\psi_3 =0$. In fact,  \eqref{approx2} has an exact harmonic traveling wave solution
\[
\varphi = \psi e^{ikx-i\omega t} + \text{c.c.},\qquad \omega = \frac{1}{2}\left(1 + k^2 |\psi|^2\right)\sgn k.
\]
The coefficient $\psi_3$ is nonzero for $0<\alpha<2$, and presumably there is no simple explicit solution
for the traveling waves in that case.

If $0<\alpha<2$, then the linearized wave motion is dispersive, and the NLS approximation for \eqref{sqg_eq} is
\begin{align*}
&\varphi(x,t) = \epsilon \psi\left(\epsilon(x - \omega_0' t),\epsilon^2t\right) e^{ikx-i\omega_0 t} +\text{c.c.}
+ \O(\epsilon^3) \qquad \text{as $\epsilon \to 0$},
\end{align*}
where a prime denotes the derivative with respect to $k$ and $\psi(X,T)$ satisfies
\[
i \psi_T = -\frac{1}{2} \omega_0'' \psi_{XX} + \sigma_2 |\psi|^2 \psi.
\]
The same NLS equation follows from the full equation \eqref{nonconseqn},
since it depends only on the cubic part the nonlinearity.

From  \eqref{defw0}, we have for $k>0$ that $\sigma_2 > 0$, and
\begin{equation*}
\omega_0'' = \begin{cases}
b_\alpha (2-\alpha)(1-\alpha) k^{-\alpha} & \text{if $0<\alpha < 1$ or $1 < \alpha < 2$},
\\
-2/k& \text{if $\alpha = 1$ (SQG)}.
\end{cases}
\end{equation*}
Equation \eqref{a-const} implies that  $b_\alpha > 0$ for $1<\alpha<2$ and $b_\alpha <0$ for $0<\alpha<1$, so
$\omega_0'' < 0$. Hence, $\omega_0'' \sigma_2 < 0$, and the NLS equation is focusing for
all $0< \alpha < 2$. It follows that  small-amplitude wavetrains on the front are modulationally unstable, and the
front supports envelope solitons.

\section{Numerical solutions}
\label{sec:num}

In this section, we show two numerical solutions of the initial value problem
for the approximate SQG front equation in \eqref{sqgivp} that indicate the formation of singularities in finite time.

The first solution is for the initial data
\begin{equation}
\varphi_0(x) = \cos(x+\pi) + \frac{1}{2}\cos[2(x+\pi +2\pi^2)].
\label{two_cos_ic}
\end{equation}
A surface plot of the solution, computed using a pseudo-spectral method with spectral viscosity, is shown in
Figure~\ref{fig:two_cos_surf}. Numerical results suggest that an oscillatory singularity forms at
$t\approx 0.06$ near $x\approx 2.15$, before there is an appreciable change in the global shape of the solution. The solution appears to be
smooth before the singularity forms, and the numerical singularity formation time does not appear to change under further refinement.
Moreover the structure of the solution remains similar as one increases the number of Fourier modes, although the number of oscillations and the $x$-location of their left endpoint increases.

One might conjecture that the formation of singularities in the approximate SQG front equation is associated with the breaking and filamentation of the front, rather than a loss of smoothness, but since we are using a graphical description of the front, we are unable to distinguish between the two. The numerical
solutions suggest that it may be possible to continue smooth solutions of  \eqref{sqgivp} by some type of weak solution after singularities
form. These weak solutions appear to remain continuous, which could be associated with the extreme thinness of any filaments
that form, as seems to occur in the case of the filamentation of vorticity fronts \cite{BiHu, BiHu1}.

In Figure~\ref{fig:sech}--\ref{fig:sech_detail}, we show a solution of \eqref{sqgivp} with the initial data
\begin{equation}
\varphi_0(x) = \sech^2\left[\frac{5(x-\pi)}{2}\right].
\label{sech_ic}
\end{equation}
for $0\le t \le 0.05$. The singularity formation time is $t\approx 0.02$. As in the previous case, a singularity forms before there
is an appreciable change in the global shape of the solution, but in this case singularities form at two different locations, the first
near the peak of the pulse and then, a little later,  a second near the front of the pulse.

\begin{figure}
\includegraphics[width=0.7\textwidth]{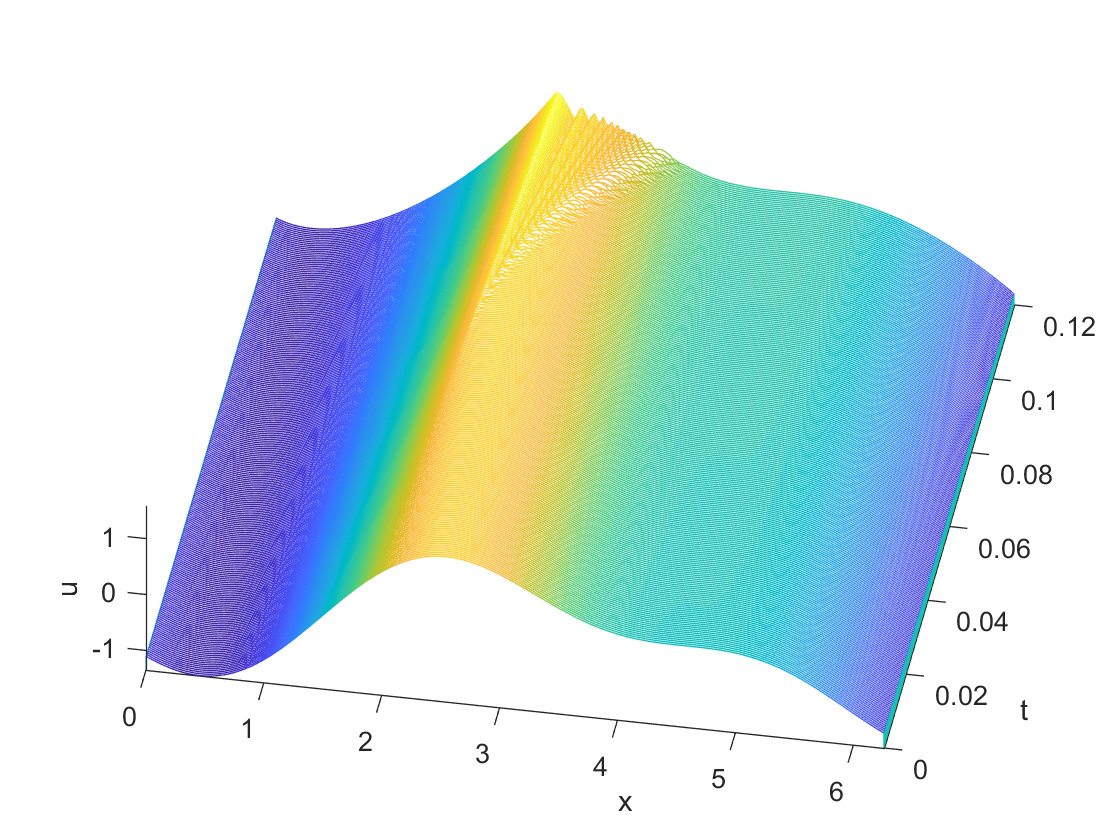}
\caption{A surface plot of the solution of \eqref{sqgivp} with initial data \eqref{two_cos_ic} for
$0\le t \le 0.12$. The solution is computed by a pseudo-spectral method with $2^{14}$ Fourier modes.
A small oscillatory singularity forms at
$t\approx 0.06$ near $x\approx 2.15$}
\label{fig:two_cos_surf}
\end{figure}

\begin{figure}
\includegraphics[width=0.7\textwidth]{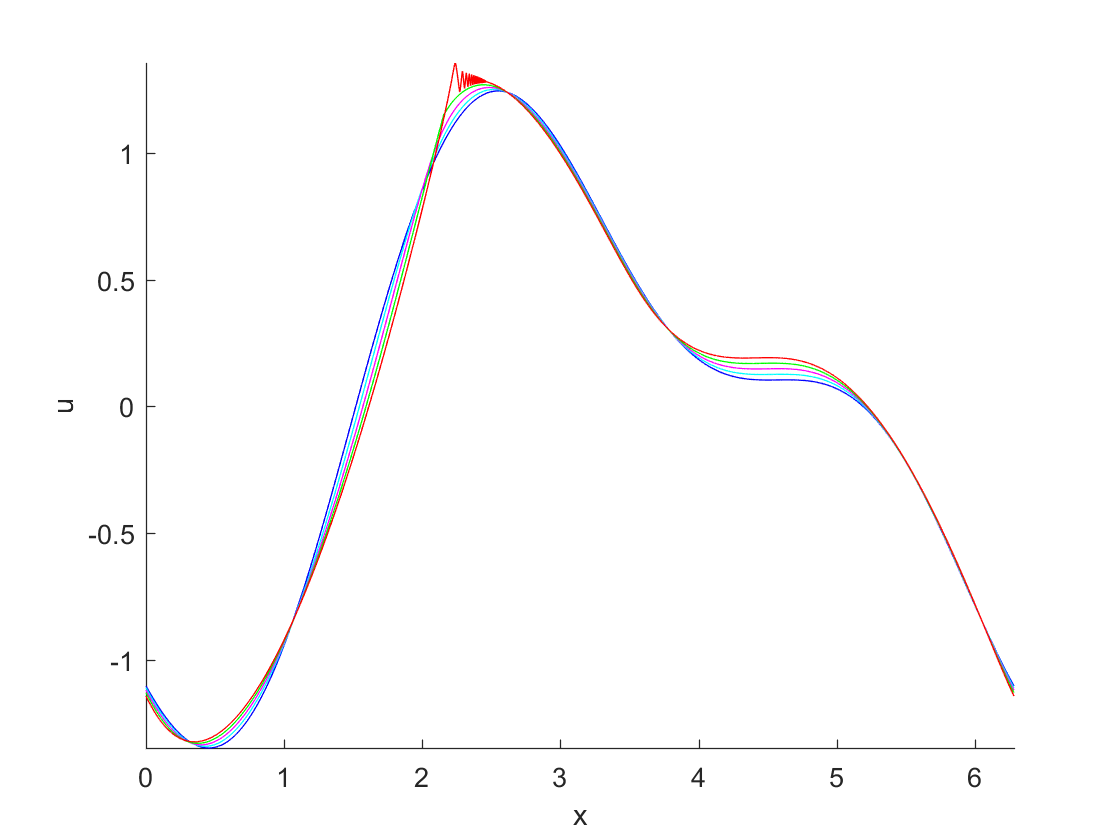}
\caption{Graphs of the solution of \eqref{sqgivp} with initial data \eqref{two_cos_ic} for
 The solution is shown at
$t=0$ (blue), $t= 0.1875$ (cyan), $t= 0.375$ (magenta), $t= 0.5625$ (green), $t = 0.75$ (red). The solution is computed by a pseudo-spectral method with $2^{15}$ Fourier modes.}
\label{fig:two_cos_graph}
\end{figure}

\begin{figure}
\includegraphics[width=0.7\textwidth]{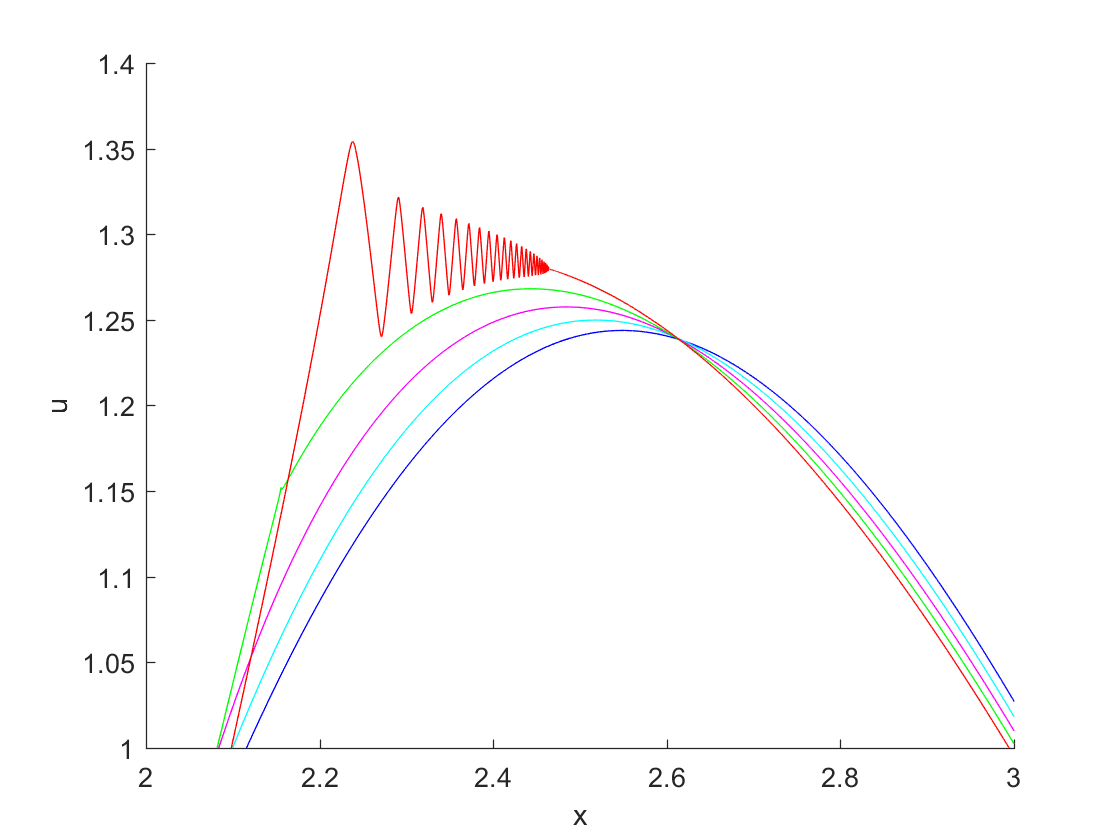}
\caption{Detail of singularity formation in the solution of \eqref{sqgivp} with initial data \eqref{two_cos_ic}
shown in Figure~\ref{fig:two_cos_graph}. The solution is shown at
$t=0$ (blue), $t= 0.1875$ (cyan), $t= 0.375$ (magenta), $t= 0.5625$ (green), $t = 0.75$ (red).}
\label{fig: two_cos_detail}
\end{figure}

\begin{figure}
\includegraphics[width=0.7\textwidth]{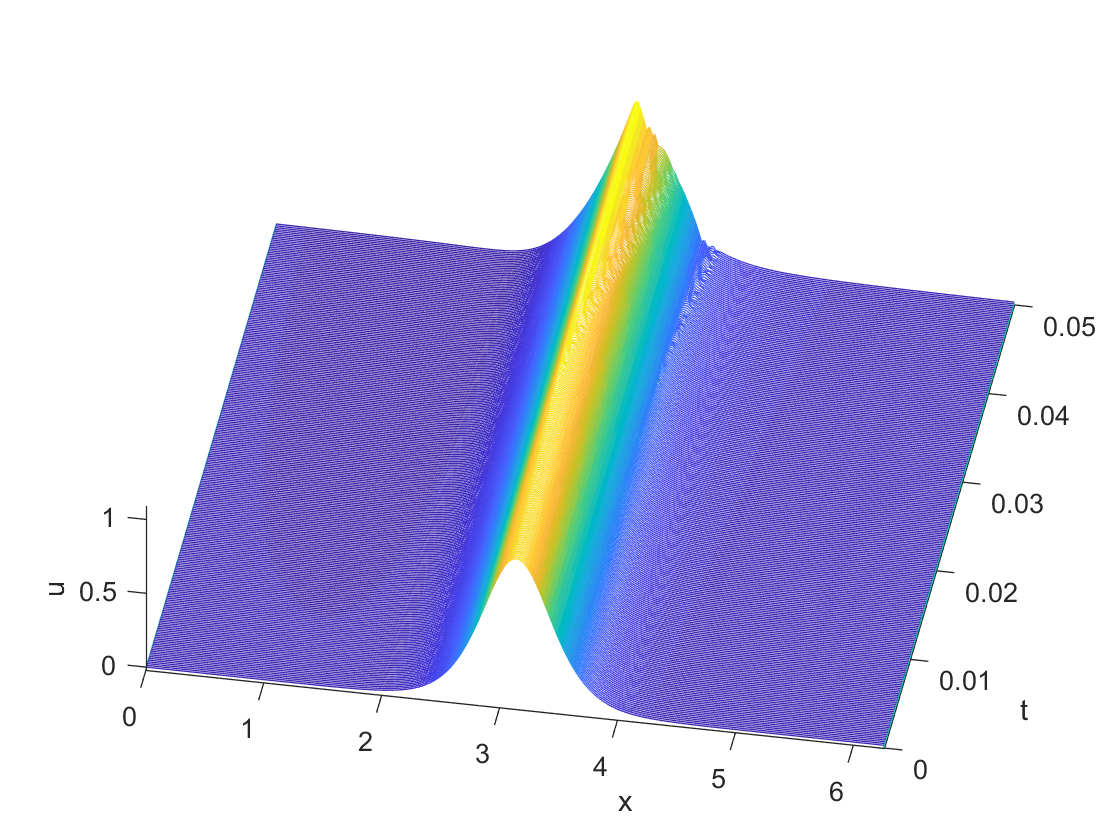}
\caption{A surface plot of the solution of \eqref{sqgivp} with initial data \eqref{sech_ic} for
$0\le t \le 0.05$. The solution is computed by a pseudo-spectral method with $2^{15}$ Fourier modes.}
\label{fig:sech}
\end{figure}

\begin{figure}
\includegraphics[width=0.7\textwidth]{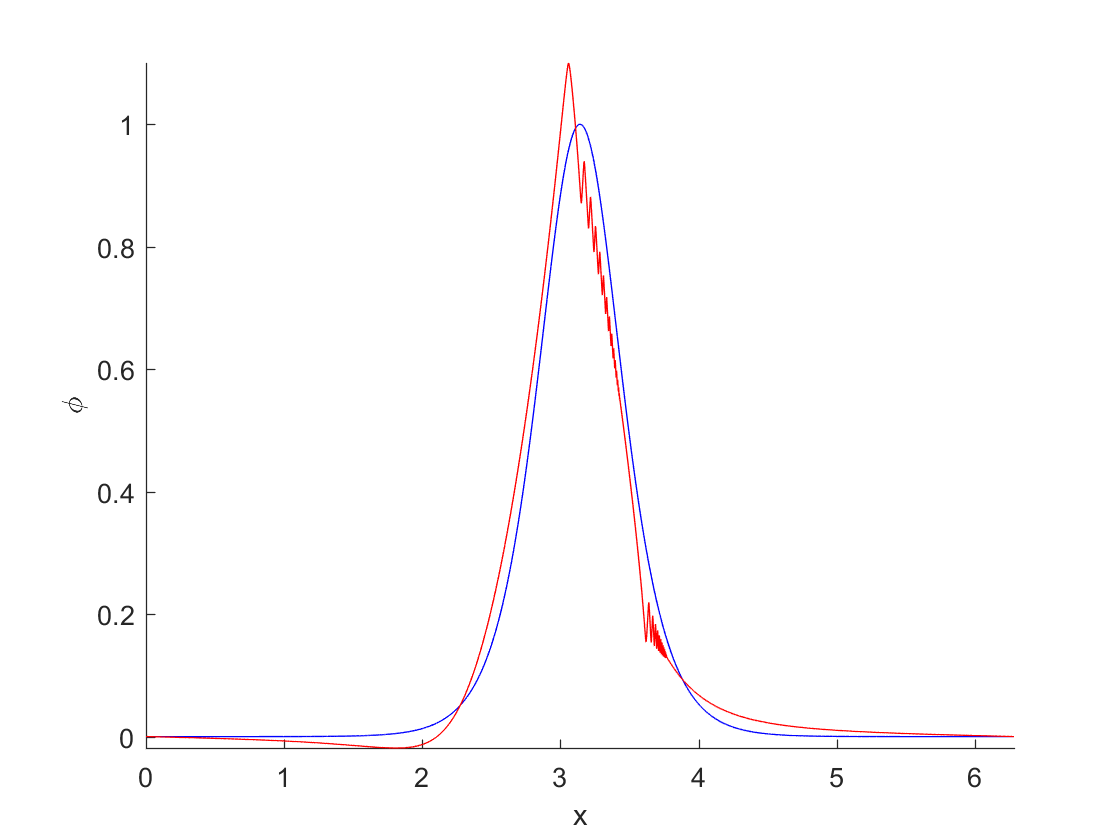}
\caption{Graphs of the solution of \eqref{sqgivp} with initial data \eqref{sech_ic} for
$t=0$ (blue) and $t=0.5$ (red). The solution is computed by a pseudo-spectral method with $2^{15}$ Fourier modes.}
\label{fig:sech_graph}
\end{figure}

\begin{figure}
\includegraphics[width=0.7\textwidth]{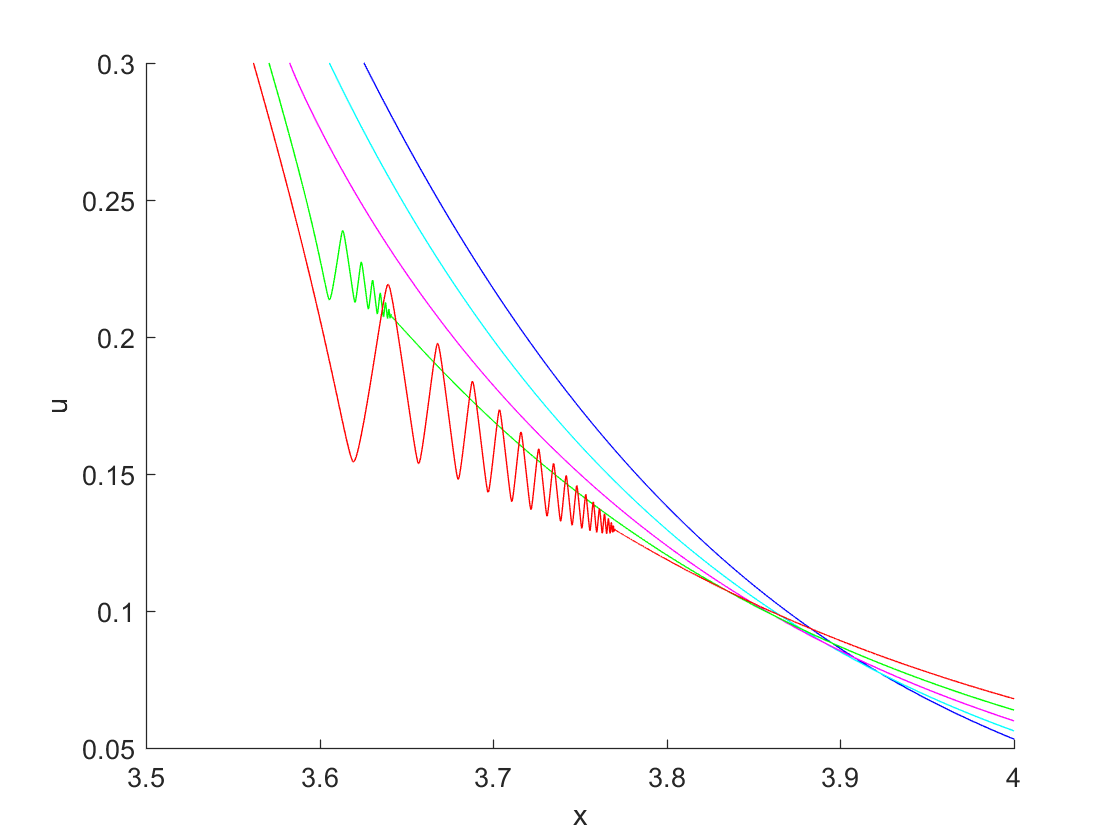}
\caption{Detail of the singularity formation near the front of the pulse
in the solution of \eqref{sqgivp} with initial data \eqref{sech_ic}
shown in Figure~\ref{fig:two_cos_graph}. The solution is shown at
$t=0$ (blue), $t= 0.125$ (cyan), $t= 0.25$ (magenta), $t= 0.375$ (green), $t = 0.5$ (red).}
\label{fig:sech_detail}
\end{figure}

\appendix
\section{Some Algebraic Inequalities}
\label{sec:ineq}

In this section, we prove the inequalities used in the local well-posedness proofs.
We use $\{k_1, k_2, k_3, k_4\}$ to denote a quadruple of real numbers such that
\[
k_1 + k_2 + k_3 + k_4 = 0,
\]
and, as in \eqref{defkm2}--\eqref{defkm1}, we denote by $(m_1, m_2, m_3, m_4)$
a permutation of $(k_1, k_2, k_3, k_4)$ such that
\[
\abs{m_1} \geq \abs{m_2} \geq \abs{m_3} \geq \abs{m_4}.
\]
If, as we assume, the $k_j$ are not identically zero, then $m_1, m_2\ne 0$, and we define
\begin{equation}
x = -\frac{m_2}{m_1},\qquad y = -\frac{m_3}{m_1}, \qquad 1 -x -y = -\frac{m_4}{m_1}.
\label{defxy}
\end{equation}
Since $m_1+m_2+m_3+m_4 = 0$, the ordering of the $|m_j|$ implies
that $0\le y \le x \le 1$ and $ |1-x-y| \le y$, so $(x,y) \in {R}$, where
the feasible region
\begin{equation}
{R} = \left\{(x, y) \in \R^2 \mid \text{$0\le y \leq x\le 1$ and $1\le x + 2y\le 2$}\right \}
\label{defregion}
\end{equation}
is shown in Figure \ref{fig:region}.
We note that $m_3 =0$ corresponds to the point $(x,y) = (1,0)$, and $m_4=0$ corresponds to the line
$x+y=1$.
The ratio $m_4/m_1$ changes sign across this line: if $x+y >1$, then $m_1$,$m_4$ have the same sign
and the opposite sign to
$m_2$, $m_3$; while if $x+y<1$, then $m_2$, $m_3$, $m_4$ have the same sign and the opposite sign to $m_1$.

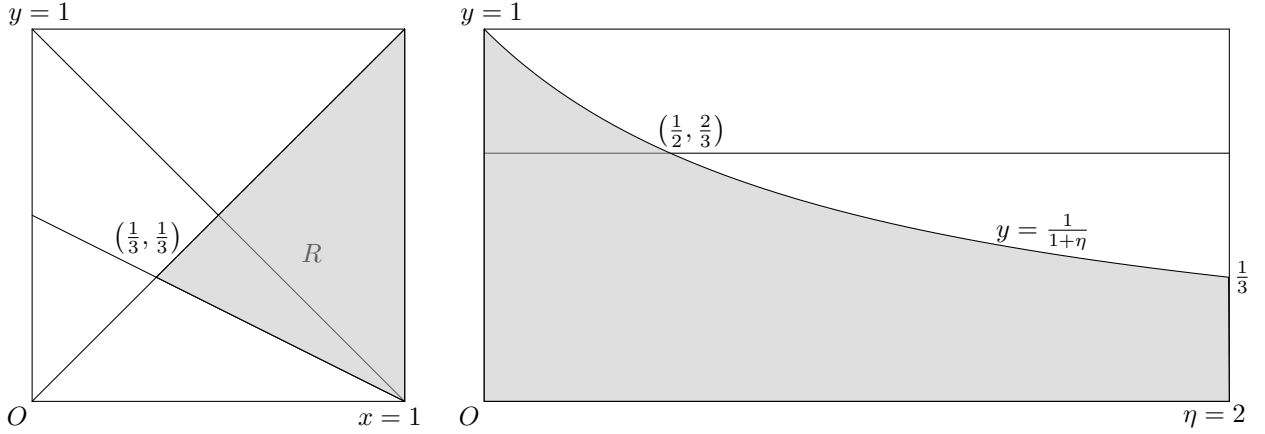
\begin{figure}[h]
\begin{tikzpicture}[line cap=round,line join=round,>=triangle 45,x=1.95in,y=1.95in]
\clip(-.07,-.07) rectangle (1.05,1.15);
\draw[color=black] (0,0) -- (1,0);
\draw[color=black] (0,1) -- (1,1);
\draw[color=black] (0,0) -- (0,1);
\draw[color=black] (1,1) -- (1,0);
\draw[color=black] (1,1) -- (0,0);
\draw[color=black] (0,.5) -- (1,0);
\draw[color=black] (0,1) -- (1,0);
\draw[color=black] (3/4,2/5) node {$R$};
\draw[color=black] (.96, -.04) node {$x = 1$};
\draw[color=black] (0.02, 1.04) node {$y = 1$};
\draw[color=black] (-.04, -.04) node {$O$};
\draw[color=black] (.31, .44) node {$\left(\frac{1}{3}, \frac{1}{3}\right)$};
\draw[fill=gray!50!white,fill opacity=0.5](1/3,1/3)--(1,0)--(1,1)--(1/3,1/3);
\end{tikzpicture}
\quad
\begin{tikzpicture}[line cap=round,line join=round,>=triangle 45,x=1.95in,y=1.95in]
\clip(-.07,-.07) rectangle (2.05,1.15);
\draw[color=black] (0,0) -- (2,0);
\draw[color=black] (0,1) -- (2,1);
\draw[color=black] (0,0) -- (0,1);
\draw[color=black] (2,1) -- (2,0);
\draw[color=black] (0,2/3) -- (2,2/3);
\draw[fill=gray!50!white,fill opacity=0.5] plot[smooth, samples=200, domain=0:2] ({\x},{1/(1+\x)}) -- (2,0) -- (0,0) -- cycle;
\draw[color=black] (2+.035,1/3) node {$\frac 13$};
\draw[color=black] (1/2+.055, 2/3+.06) node {$\left(\frac 12, \frac 23\right)$};
\draw[color=black] (1.96, -.04) node {$\eta = 2$};
\draw[color=black] (1.5, 0.455) node {$y = \frac{1}{1 + \eta}$};
\draw[color=black] (0.02, 1.04) node {$y = 1$};
\draw[color=black] (-.04, -.04) node {$O$};
\end{tikzpicture}
\centering
\caption{Left: The feasible region $R$ for $(x,y)$-variables in \eqref{defregion}. Right: The feasible region for $(\eta,y)$-variables in
\eqref{polarco}.}
\label{fig:region}
\end{figure}

We begin with the following inequality for a symmetric function of fractional powers.

\begin{lemma}
\label{2s+1ineq}
If $k_j, m_j \in \R$ with $j = 1, 2, 3, 4$ are defined as above, then for every $s > 0$
there exists a constant $C_0(s)$, depending only on $s$, such that
\[
\abs{k_1 \abs{k_1}^{2s} + k_2 \abs{k_2}^{2s} + k_3 \abs{k_3}^{2s} + k_4 \abs{k_4}^{2s}}
\leq C_0(s) \abs{m_1}^s \abs{m_2}^s \abs{m_3}.
\]
\end{lemma}

\begin{proof}
Both sides of the inequality are zero if $m_3=0$,  when $m_1=-m_2$ and $m_4=0$, so
we may assume that $m_3 \ne 0$.
Using \eqref{defxy}, and the fact that the $k_j$ are a permutation of the $m_j$, we get that
\[
\frac{k_1 \abs{k_1}^{2s} + k_2 \abs{k_2}^{2s} + k_3 \abs{k_3}^{2s} + k_4 \abs{k_4}^{2s}}
{ \abs{m_1}^s \abs{m_2}^s m_3}
= f(x,y),
\]
where the continuous function $f \colon {R}\setminus\{(1,0)\}\to \R$ is given by
\[
f(x, y) = \frac{1 - x^{2s+1} - y^{2s+1} + (x + y - 1) \abs{x + y - 1}^{2s}}{x^s y}.
\]
The only place where $f$ could fail to be bounded is near $(1, 0)$.
Writing
\begin{equation}
x = 1 - \eta y, \qquad \text{with $0 \leq \eta \leq 2$},
\label{polarco}
\end{equation}
and Taylor expanding $f$ as $y \to 0^+$, we get that
\[\begin{aligned}
f(1 - \eta y, y)
& = (2s + 1) \eta + \O(y + y^{2s})\end{aligned}\]
uniformly in $0\le \eta\le 2$. It follows that
\[
\lim_{y \to 0^+}\sup_{0\le\eta\le 2} \abs{f(1 - \eta y,y)} = 2\cdot (2s + 1),
\]
which proves the Lemma.
\end{proof}

Numerical computations show that the supremum of $|f|$ on ${R}\setminus\{(1,0)\}$
is attained at $(x,y) = (1/3,1/3)$ if $s\ge s_0$, where $s_0 \approx 0.6365$ is the positive value of $s$ at which  $3^{s+1} - 3^{1-s}= 2(2s+1)$. In that case, we may take
\begin{equation}
C_0(s) = 3^{s+1} - 3^{1-s}.
\label{numC0}
\end{equation}

Next, we estimate the gSQG and SQG kernels defined in \eqref{Skernel1a2} and \eqref{Skernelsqg}. From \eqref{defS}, these kernels have the form
\begin{align}
S(k_1, k_2, k_3, k_4)
&= |m_1|^{3-\alpha}\left[g\left(-\frac{m_2}{m_1},-\frac{m_3}{m_1}\right) + h\left(-\frac{m_2}{m_1},-\frac{m_3}{m_1}\right)\right]
\label{defSgh}
\end{align}
where $\alpha = 1$ in the case of the SQG kernel, and
\begin{align}
\begin{split}
g(x,y) &=  a(y) + a(x+y-1) - a(1-x),
\\
h(x,y) &= a(1) + a(x)- a(1 - y) - a(x+y).
\end{split}
\label{defgh}
\end{align}
First, we estimate $h$.

\begin{lemma}
\label{lemhest}
Let $a$ be given by \eqref{defsqga} or \eqref{sqgalphaa} with $\alpha \le 2$, and
let $h$ be given by \eqref{defgh}.
There exists $C > 0$ such that
\[
|h(x,y)| \le C |x+y-1| y\qquad \text{for all $(x,y)\in R$}.
\]
\end{lemma}

\begin{proof}
Using coordinates \eqref{polarco}, we have
\begin{align}
\begin{split}
h(1-\eta y,y)  &= a(1) + a(1 - \eta y)   - a(1 - y) - a(1+(1-\eta)y)
\\
&= - \int_{1-y}^1 \int_0^{(1-\eta)y} a''(t+s) \diff{s}\diff{t}.
\end{split}
\label{int_ts}
\end{align}
If  $0\le y \le {2}/{3}$ and $(x,y)\in R$, then ${1}/{3} \le t+s \le {5}/{3}$ in \eqref{int_ts}.
Since $|a''(x)|\le M$ is bounded on this interval, we get that
\begin{align*}
|h(1-\eta y,y)|  &\le M|1-\eta| y^2.
\end{align*}
If ${2}/{3}\le y \le 1$ and $(x,y)\in R$, then
$0\le \eta \le {1}/{2}$, and it follows that
\begin{align*}
|h(1-\eta y,y)| &\le \int_0^1 \int_0^{1} |a''(t+s)| \diff{s}\diff{t}
\\
&\le  \frac{9}{2}\left(\int_0^1 \int_0^{1} |a''(t+s)| \diff{s}\diff{t}\right) |1-\eta| y^{2}.
\end{align*}
Since the integral converges and $(1-\eta)y = x+y-1$, we obtain the result.
\end{proof}

\begin{lemma}
\label{kernelest1a2}
Let $S$ be given by \eqref{Skernel1a2} with $1 < \alpha \le 2$.
If $k_j, m_j \in \R$ with $j = 1, 2, 3, 4$ are defined as above, then
 there exists a constant $C_1(\alpha)$, depending only on $\alpha$, such that
\[
|S(k_1, k_2, k_3, k_4)|
\leq C_1(\alpha) \abs{m_3}^{2-\alpha} \abs{m_4}.
\]
\end{lemma}

\begin{proof}
It follows from \eqref{Skernel1a2} that $S(k_1,k_2,k_3,k_4) = 0$ whenever any of the $k_j$ vanishes,
so we may assume that $m_4\ne 0$. The kernel $S$ is given by \eqref{defSgh}--\eqref{defgh} with $a(x) = |x|^{3-\alpha}$, where we neglect the unimportant constant factor $c_\alpha$ in \eqref{Skernel1a2}. We have
\begin{align*}
|g(1-\eta y,y)| &=\left[1 - \eta^{3-\alpha} + |1-\eta|^{3-\alpha}\right]y^{3-\alpha}
\le C |1-\eta|y^{3-\alpha},
\end{align*}
where $C$ denotes a generic constant depending on $\alpha$.
Using this inequality, \eqref{polarco},  and Lemma~\ref{lemhest}, we get
\begin{align*}
|g(x,y) + h(x,y)| \le C |x+y-1| y^{2-\alpha}\left[1 + y^{\alpha-1}\right]
\le 2C |x+y-1| y^{2-\alpha},
\end{align*}
and the use of this inequality in \eqref{defSgh} proves the Lemma.
\end{proof}

This estimate in Lemma~\ref{kernelest1a2} fails for $0<\alpha < 1$ because of the term involving $h$. Instead, we get
\[
\left|S(k_1, k_2, k_3, k_4)\right| \le C|m_1|^{1-\alpha} |m_3| |m_4|,
\]
and $S$ is not bounded the smaller wavenumbers.

Finally, we estimate the SQG kernel, where an additional logarithmic growth factor appears.

\begin{lemma}
\label{kernelestsqg}
Let $S$ be given by \eqref{Skernelsqg}. If $k_j, m_j \in \R\setminus\{0\}$ with $j = 1, 2, 3, 4$ are defined as above,
then there exists a numerical constant $C_2$ such that
\[
|S(k_1, k_2, k_3, k_4)| \leq C_2 \abs{m_3} \abs{m_4} \log\left(1 + \abs{\frac{m_2}{m_3}}\right).
\]
\end{lemma}

\begin{proof}
The kernel $S$ is given by \eqref{defSgh}--\eqref{defgh} with $a(x) = -x^2\log|x|$. We have
\begin{align*}
\left|g(1-\eta y,y)\right| &=  \left|y^2 \log y + (1-\eta)^2y^2 \log|(1-\eta)y| - \eta^2 y^2 \log \eta y\right|
\\
&= \left|2(1-\eta) y^2\log y  +\left[(1-\eta)^2 \log|1-\eta| -\eta^2 \log\eta\right] y^2\right|
\\
& \le  C |1-\eta|y^2\left[1+  \log(1/y)\right].
\end{align*}
Since $x\ge 1/3$ and $x/y\ge 1$, it follows that
\[
|g(x,y)| \le C|x+y-1| y \log\left(1 + \frac{x}{y}\right).
\]
Using this inequality  and Lemma~\ref{lemhest}, we get that
\[
|g(x,y) + h(x,y)| \le C |x+y-1| y \log\left(1 + \frac{x}{y}\right),
\]
and the use of this inequality in \eqref{defSgh} proves the Lemma.
\end{proof}

Numerical computations show that in Lemma~\ref{kernelestsqg} we can take, for example,
\begin{equation}
C_2=5.
\label{numC2}
\end{equation}
The worst case for the growth of $S$ is when two wavenumbers are in the same ``shell''
with much larger and almost equal absolute values than the other two wavenumbers, which
happens near the point $(x,y)=(1,0)$ in $R$. For example, suppose that
\[
k_1 = k+a,\qquad k_2 = -(k+b),\qquad k_3 = -a,\qquad k_4 = b,
\]
and consider the limit $k\to \infty$ with $a, b >0$ fixed. Then one finds that
\begin{align*}
S(k_1,k_2,k_3,k_4)  &= -2ab \log|k| +\O(1)
=2 m_3 m_4\log \left|\frac{m_2}{m_3}\right| + \O(1).
\end{align*}
Thus, the logarithmic factor in Lemma~\ref{kernelestsqg} cannot be improved upon.

We end this section with a Corollary of Lemma \ref{kernelestsqg} for the SQG kernel as a function
of integer wavenumbers. This Lemma uses the fact that the $|k_j|$ are bounded away from zero,
so it does not apply in the spatial case with $k_j\in \R\setminus\{0\}$.

\begin{corollary}
\label{kernelestsqg-int}
Let $S$ be given by  \eqref{Skernelsqg}. If $k_j, m_j \in \Z_*$ with $j = 1, 2, 3, 4$
are defined as above, then there exists a constant $C_2$
such that
\begin{align*}
\left| S(k_1, k_2, k_3, k_4)\right| &\le
C_2\abs{m_3} \abs{m_4}\left[\log(1 + |m_1|) \log(1+ |m_2|)\right]^{1/2}.
\end{align*}
\end{corollary}

\begin{proof}
The result follows immediately from Lemma \ref{kernelestsqg}, since $\abs{m_2} \le \abs{m_1}$ and $\abs{m_3} \geq 1$.
\end{proof}

\end{document}